\documentclass[11pt]{amsart}

\usepackage[utf8x]{inputenc}
\usepackage[colorlinks=false, linktocpage=true]{hyperref}
\usepackage{fullpage}
\usepackage{color}
\usepackage{longtable}
\usepackage{amscd,amssymb,amsthm}
\usepackage[all]{xy}

\usepackage{tikz-cd}

\usepackage{mathrsfs}
\usepackage{txfonts}

\allowdisplaybreaks[4]

\numberwithin{equation}{section}

\newcommand{\bC}{\mathbb{C}}

\newcommand{\bQ}{\mathbb{Q}}
\newcommand{\bR}{\mathbb{R}}
\newcommand{\bV}{\mathbb{V}}
\newcommand{\bZ}{\mathbb{Z}}

\newcommand{\C}{\mathbb{C}}
\newcommand{\R}{\mathbb{R}}
\newcommand{\Z}{\mathbb{Z}}

\newcommand{\Q}{\mathbb{Q}}

\newcommand{\cH}{\mathcal{H}}

\newcommand{\cN}{\mathcal{N}}

\newcommand{\cT}{\mathcal{T}}

\newcommand{\Hom}{\mathrm{Hom}}

\renewcommand{\deg}{\mathrm{deg}}
\newcommand{\Pic}{\mathrm{Pic}}

\newcommand{\rank}{\mathrm{rank}}

\newcommand{\Ad}{\mathrm{Ad}}
\newcommand{\ad}{\mathrm{ad}}

\newcommand{\Spec}{\mathrm{Spec}}
\newcommand{\Mor}{\mathrm{Mor}}

\newcommand{\Lie}{\mathrm{Lie}}
\newcommand{\diag}{\mathrm{diag}} 
\newcommand{\ab}{ {\mathrm{ab}}  } 

\newcommand{\U}{\mathbf{U}}
\renewcommand{\O}{\mathbf{O}}
\newcommand{\GL}{\mathbf{GL}}
\newcommand{\SL}{\mathbf{SL}}

\newcommand{\SO}{\mathbf{SO}}
\newcommand{\Sp}{\mathbf{Sp}}

\newcommand{\cC}{\mathscr{C}}
\newcommand{\cG}{\mathscr{G}}
\newcommand{\cP}{\mathscr{P}}

\newcommand{\fP}{\mathfrak{P}}
\newcommand{\fT}{\mathfrak{T}}
\newcommand{\fa}{\mathfrak{a}}
\newcommand{\fg}{\mathfrak{g}}
\newcommand{\fh}{\mathfrak{h}}

\newcommand{\ft}{\mathfrak{t}}
\newcommand{\fb}{\mathfrak{b}}
\newcommand{\fp}{\mathfrak{p}}

\newcommand{\fu}{\mathfrak{u}} 
\newcommand{\fz}{\mathfrak{z}}
\newcommand{\fl}{\mathfrak{l}}
\newcommand{\fM}{\mathfrak{M}}
\newcommand{\fBun}{\mathfrak{Bun}}

\newcommand{\ep}{\varepsilon}

\renewcommand{\phi}{\varphi}
\newcommand{\Ga}{\Gamma}

\newcommand{\tal}{ {\tilde{\alpha}} }

\newcommand{\hLambda}{\widehat{\Lambda}}

\newcommand{\ModXrd}{\mathscr{M}_X^{r,d}}
\newcommand{\ModXrL}{\mathscr{M}_X^{r,\mathscr{L}}}
\newcommand{\ModGXdss}{\mathscr{M}_{G_X}^{\,d,ss}}
\newcommand{\ModGXdst}{\mathscr{M}_{G_X}^{\,d,s}}

\newcommand{\CP}{\mathbb{C}\mathbf{P}}

\newcommand{\ov}[1]{\overline{#1}}

\newcommand{\lra}{\longrightarrow}
\newcommand{\lmt}{\longmapsto}

\newcommand{\Ker}{\mathrm{Ker}\,}
\renewcommand{\Im}{\mathrm{Im}\,}

\newcommand{\piG}{\pi_1G}
\newcommand{\piL}{\pi_1L}
\newcommand{\piH}{\pi_1H}
\newcommand{\BunGX}{\fBun_{\, G_X}}

\newcommand{\BunGXd}{\fBun_{\, G_X}^{\ d}}
\newcommand{\BunGXdss}{\fBun_{\, G_X}^{\,d,ss}}
\newcommand{\BunGXdst}{\fBun_{\, G_X}^{\,d,s}}
\newcommand{\BunLXmuss}{\fBun_{\, L_X}^{\,\ud_\umu,ss}}
\newcommand{\LGd}{\,I(G,d)}

\newcommand{\Bmu}{\fBun_{\nu}}
\newcommand{\Bnu}{\fBun_{\,\nu}}
\newcommand{\ud}{\delta}
\newcommand{\umu}{\mu}
\newcommand{\Blnu}{\fBun_{\,\leq\nu}}
\newcommand{\Bslnu}{\fBun_{\,<\nu}}
\newcommand{\Sch}{(\mathrm{Sch}/\bC)}
\newcommand{\Gp}{\mathrm{Groupoids}}

\renewcommand{\geq}{\geqslant}
\renewcommand{\leq}{\leqslant}
\renewcommand{\i}{\sqrt{-1}}

\theoremstyle{plain}
\newtheorem{dummy}{dummy}[section]
\newtheorem{lemma}[dummy]{Lemma}
\newtheorem{theorem}[dummy]{Theorem}

\newtheorem{corollary}[dummy]{Corollary}
\newtheorem{proposition}[dummy]{Proposition}
\newtheorem{example}[dummy]{Example}

\theoremstyle{definition}
\newtheorem{definition}[dummy]{Definition}

\newtheorem*{ack}{Acknowledgments}

\newtheorem{remark}[dummy]{Remark}

\theoremstyle{plain}
\newtheorem{thm}{Theorem}

\begin{document}

\title[Hodge numbers of moduli  of $G$-bundles]{Hodge numbers of moduli of principal bundles on a curve}
\author{Chiu-Chu Melissa Liu}
\address{Chiu-Chu Melissa Liu, Department of Mathematics, Columbia University, New York City, NY, USA.}
\email{ccliu@math.columbia.edu}
\author{Florent Schaffhauser}
\address{Florent Schaffhauser, Department of Mathematics, Heidelberg University, Germany}
\email{fschaffhauser@mathi.uni-heidelberg.de}


\date{\today}

\begin{abstract}
We prove an inversion theorem for recursive formulas satisfied by certain families of converging power series in two variables. These power series are indexed by the Harder-Narasimhan types of principal $G$-bundles of degree $d \in \piG$ on a smooth projective curve $X$, where $G$ is a connected complex reductive group. As an application, we obtain a closed formula for the Hodge-Poincaré series of moduli stacks of semistable principal $G$-bundles of degree $d$. We also compute the variation of Hodge structure of the moduli stack of all principal $G$-bundles over $X$, as a function of the period matrix of that curve.
\end{abstract}

\maketitle

\tableofcontents

\section{Introduction and main results}

Let $X$ be a smooth complex projective curve of genus $g \geqslant 2$ and let $G$ be a connected complex reductive group. Moduli spaces of principal $G$-bundles on $X$ were constructed by Ramanathan in \cite{Ramanathan}. For every $d \in \piG$, there exists a connected moduli space of semistable principal $G$-bundles, which is a projective variety of dimension $(g-1) \dim(G) + \dim Z_G$, where $Z_G$ is the center of $G$, and whose points are in bijection with $S$-equivalence classes of such bundles. Ramanathan's semistability condition generalizes Mumford's condition for vector bundles on a curve (see \cite{Mumford_ICM}), for which Narasimhan and Seshadri constructed a moduli space  in \cite{NS}. 

\smallskip

The Betti numbers of moduli spaces of semistable principal $G$-bundles were computed by Atiyah and Bott in \cite{AB}. More precisely, they computed the Poincaré series of moduli stacks of such bundles, as well as the Poincaré series of the moduli stack of all principal $G$-bundles, using a gauge-theoretic approach. For curves defined over finite fields and in the case when $G = \GL_n$, a formula for these Poincaré series had been obtained earlier by Desale and Ramanan in \cite{DeRa}, following the number-theoretic approach of Harder and Narasimhan \cite{HN}. The Harder-Narasimhan reduction of a principal bundle plays a central role in the Atiyah-Bott approach (see Section \ref{HN_stratif_section} for a review) and the formula for the Poincaré polynomial obtained by Atiyah and Bott is recursive: to compute the Betti numbers of moduli stacks of semistable principal $G$-bundles, one needs to know the Betti numbers of moduli stacks of semistable principal $L$-bundles, where $L$ ranges over Levi subgroups of $G$, hence is a reductive of smaller dimension than $G$. In the case when $G = \GL_n$, Zagier \cite{Zagier} was able to invert the recursion and obtain a closed formula for the Poincaré series of moduli stacks of semistable principal $G$-bundles of rank $r$ and degree $d \in \Z$. When $r$ and $d$ are coprime, one can deduce recursive and closed formulas for the Poincaré polynomials of the coarse moduli spaces of those stacks, which is the moduli space of semistable vector bundles constructed by Narasimhan and Seshadri.

\smallskip

Going back to the case of principal bundles, Laumon and Rapoport found a completely general way to invert the Atiyah-Bott recursive formula, valid for all reductive groups. The formula given in \cite{LR} holds for all connected reductive groupes $G$ such that $[G,G]$ is simply connected, and this restriction was lifted by Ho and Liu in \cite{Ho_Liu_YM_AMS}. In \cite{Teleman}, Teleman computed the Hodge numbers of moduli stacks of principal $G$-bundles for all semisimple group $G$ and, in the case when $G = \GL_n$, Earl and Kirwan obtained a recursive formula for the Hodge-Poincaré series of moduli stacks of semistable vector bundles in \cite{Earl_Kirwan}. Our goal in this paper is to generalize the above in two directions and obtain the following results:
\begin{enumerate}
	\item A recursive formula for Hodge-Poincaré series of moduli stacks of semistable principal $G$-bundles.
	\item An inversion theorem for such recursive formulas, resulting in a closed formula for Hodge-Poincaré series of moduli stacks of semistable principal $G$-bundles (Theorem \ref{thm:HP}). 
\end{enumerate}
We now state our main result. 
%
\begin{thm}[Theorem \ref{thm:HP}] \label{thm:HP_intro}
	For all $d\in\piG$, the Hodge-Poincaré series of the moduli stack $\BunGXdss$ is equal to the expansion in $\Z[\![u,v]\!]$ of the element of $\Q(u,v)$ defined by the right-hand side of the following formula:
	$$
	HP_{u,v}(\BunGXdss)  = 
	\sum_{I\subseteq \Delta}
	(-1)^{\dim Z(L^I)-\dim Z_G}
	a_{u,v}(L^I_X)
	\frac{(uv)^{(g-1)\dim U^I}}
	{\prod_{\alpha\in I}
		\big(1-(uv)^{2\varrho^I(\alpha^\vee)}\big)}\cdot
	(uv)^{2\sum_{\alpha\in I}\varrho^I(\alpha^\vee)
		\langle \varpi_\alpha(d)\rangle}
	$$
	where, for all subset $I\subset \Delta$ of the set of simple roots of $G$, we denote by:
	\begin{itemize}
		\item $P^I$ the parabolic subgroup of $G$ associated to $I\subset\Delta$, 
		\item $L^I$ the Levi quotient of $P^I$,
		\item $U^I$ the unipotent radical of $P^I$, 
		\item $a_{u,v}(L^I_X)$ the Hodge-Poincaré series 
		$$
		a_{uv}(L^I_X) =  \left(\frac{ (1+u)^g(1+v)^g }{1-uv} \right)^{\dim Z(L^I)} \prod_{k=\dim Z(L^I)+1}^{\mathrm{rk}(L^I)}
		\frac{(1+u^{d_k(L^I)} v^{d_k(L^I)-1})^g (1+u^{d_k(L^I)-1} v^{d_k(L^I)})^g}{\big(1- (uv)^{d_k(L^I)-1}\big)\big(1-(uv)^{d_k(L^I)}\big)}\,.
		$$
		of the moduli stack of all principal $L^I$-bundles on $X$ of a fixed degree $\delta\in \pi_1 L_I$, as calculated in Theorem \ref{thm:Hodge},
		\item $\varrho^I\in\fh^*$ the linear form on the Cartan sub-algebra $\fh\subset \fg$ defined by 
		$$
		\varrho^I\quad =\quad \frac{1}{2}\sum_{\tiny \begin{array}{c}\beta\in R_+\ |\ \exists\,\alpha\in I\\
				\langle \beta,\alpha^\vee \rangle >0\end{array}} \beta
		$$
	\end{itemize}
	and, moreover,
	\begin{itemize}
		\item $\varpi_\alpha(d)\in \bQ/\bZ$ is the image of $d$, as defined in Diagram \eqref{diagram_with_fund_weights}, ,
		\item and $\langle x\rangle \in\bQ$ is the unique representative of the class $x\in \bQ/\bZ$ such that $0< \langle x\rangle\leq 1$.
	\end{itemize}
\end{thm}
As we shall see in Section \ref{proof_of_inversion}, the proof of Theorem \ref{thm:HP_intro} is a refinement of the proof of the inversion theorem in one variable given in \cite[Appendix A]{Ho_Liu_YM_AMS}, generalizing it to the case of convergent series of two variables. Here is an outline of the paper.

\begin{itemize}
	\item In Section \ref{Hodge_theory_for_stacks}, we review the Hodge theory of stacks, following Totaro \cite{Totaro} and Kubrak-Prikhodko \cite{Kubrak_Prikhodko}, which slightly generalizes Teleman's result (to the reductive case), via a different approach that also enables us to compute the variation of Hodge structure of the moduli stack of all principal $G$-bundles as function of a variation of Hodge structure on $X$ (Corollary \ref{cor:VHS}).
	\item In Section \ref{recursive_formulae_section}, we obtain a recursive formula for the Hodge-Poincaré series of moduli stacks of semistable principal bundles. The main result of that section is Theorem \ref{perfection_thm}, which enables us to spell out the recursive formula in Equation \eqref{recursive_formula_for_HP_series}.
	\item In Section \ref{inversion_section}, we prove Theorem \ref{thm:HP_intro} and specialize it to the case of classical groups. The main technical step towards Theorem \ref{thm:HP_intro} is Theorem \ref{thm:inversionII}, a general inversion formula
	for a $\bZ[\![ u,v ]\!]$-module $A$ which is complete with respect to the $\bZ[\![u,v]\!]$-adic topology. This is a refined version of the inversion formula \cite[Theorem A.6]{Ho_Liu_YM_AMS} (which is a slightly modified version of \cite[Theorem 2.4]{LR}) for
	a $\bZ[\![ t ]\!]$-module $A$ which is complete with respect to the $\bZ[\![t]\!]$-adic topology.  We then apply Theorem \ref{thm:inversionII} to invert the recursive formula for the Hodge Poincar\'{e} series  (Theorem \ref{thm:recursion}, which is equivalent to 
	Theorem \ref{perfection_thm}) to obtain the closed formula stated in Theorem \ref{thm:inversionIII}, which is equivalent to our main result Theorem \ref{thm:HP_intro}. 
	The general formula of Theorem \ref{thm:HP_intro} is then specialized to the case of the classical groups in Section \ref{section:classical_groups}.

	\item In Section \ref{applications}, we derive from our main result a closed formula for the Hodge-Poincaré series of coarse moduli spaces of moduli stacks of semistable principal $G$-bundles of degree $d$, under the assumption that all semistable such bundles are in fact stable. When $G=\GL_n$, this is equivalent to the assumption that $r$ and $d$ are coprime. We note in Remark \ref{good_case} that this so-called good case cannot occur when $G$ is simply connected. We also give a few applications of our results to the vector bundle case (e.g. Theorem \ref{HP_pol_mod_spaces_of_vector_bundles}).
\end{itemize}

\begin{ack}
C.-C. M. Liu wishes to thank Kai Behrend, Daniel Halpern-Leistner, and Ezra Getzler for helpful conversations. F. Schaffhauser's research is partially supported by AEI-DFG \textit{V-SHARP} Project SCHA 2147\_1-1 AOBJ 706923. C.-C. M. Liu is partially supported by NSF DMS-1564497. 
\end{ack}

\section{Hodge theory of moduli stacks of principal bundles}\label{Hodge_theory_for_stacks}

In this section, we compute the Hodge-Poincar\'{e} series of the moduli stack of all principal $G$-bundles on a smooth projective curve, for $G$ a complex reductive group (Theorem \ref{thm:Hodge}).  When $G$ is semisimple, this was computed by C. Teleman in \cite[Chapter IV]{Teleman}.

\subsection{Hodge theory of classifying stacks}

Let $G$ be a connected, reductive algebraic group defined over $\bC$. The classifying space $BG$ of $G$ is a smooth algebraic stack. By definition, the stack $BG :=[\bullet/G]$ is the quotient stack of a point $\bullet$ by the trivial action of $G$, and the projection $\bullet \lra [\bullet/G]$ is the universal principal $G$-bundle, also denoted $EG\lra BG$. The projection $\bullet\lra BG=[\bullet/G]$ is a smooth morphism of relative dimension equal to $\dim G$, so $BG$ is a smooth algebraic stack of dimension equal to $-\dim G$. By \cite[Theorem~9.1.1]{Deligne_Hodge_III}, there is a $\Q$-Hodge structure on $H^*(BG;\C)$ which is pure and Hodge-Tate, in the following sense.

\begin{definition}
A \emph{pure $\Q$-Hodge structure} is a triple $(V_\Q,(V^{p,q})_{p,q\geq 0},\psi)$ consisting of a $\Q$-vector space $V_\Q$, a family $(V^{p,q})_{p,q\geq 0}$ of finite-dimensional $\C$-vector spaces, and an isomorphism of $\C$-vector spaces $$\psi: \bigoplus_{k=0}^{+\infty}\bigoplus_{p+q=k} V^{p,q}\overset{\simeq}{\lra} V_\Q\otimes_\Q\C $$ such that $\ov{V^{p,q}} = V^{q,p}$ as subspaces of $V_\C:=V_\Q\otimes_\Q\C$, where $v\lmt\ov{v}$ is the $\C$-antilinear automorphism of $V_\C=V_\Q\otimes_\Q\C$ induced by the complex conjugation of $\C$.  

\smallskip

The \emph{Hodge-Poincar\'{e} series} of a pure Hodge structure is the formal power series in two variables $$HP_{u,v}(V_\C) := \sum_{k=0}^{+\infty}\sum_{p+q=k} u^p\,v^q\, \dim_\C\,V^{p,q}\in\Z[\![u,v]\!].$$

A pure $\Q$-Hodge structure is said to be \emph{Hodge-Tate} (or \emph{balanced}) if $V^{p,q}=0$ when $p\neq q$. When the cohomology $H^*(\fM;\C)$ of an algebraic stack $\fM$ carries a pure Hodge structure, the Hodge-Poincar\'{e} series of that Hodge structure will be denoted by $HP_{u,v}(\fM)$.
\end{definition}

In what follows, we shall refer to a $\Q$-Hodge structure $(V_\Q,(V^{p,q})_{p,q\geq 0},\psi)$ simply as a Hodge structure on the complex vector space $V_\C$. By definition of a pure Hodge structure, an element $v\in V^{p,q}$ has a complex conjugate $\ov{v}\in V^{q,p}$. In the examples of interest to us, we have $V_\Q=H^*(\fM;\Q)$, meaning that $V_\Q$ is the cohomology, with coefficients in the locally constant sheaf $\Q$, of a smooth complex algebraic stack $\fM$. The vector spaces $V^{p,q}$ are the Hodge cohomology groups of $\fM$, which are defined as 

\begin{equation}\label{Hodge_groups}
H^{p,q}(\fM):=H^q(\fM;\Lambda^p\,\mathbb{L}_{\fM})
\end{equation} 
where $\mathbb{L}_{\fM}$ is the cotangent complex of $\fM$ \cite{Totaro}. When the Hodge-to-de-Rham spectral sequence of $\fM$ degenerates at the $E_1$ page, we have a decomposition $$H^{\,k}_{DR}(\fM) \simeq \bigoplus_{p+q=k} H^{p,q}(\fM)$$ and the de Rham isomorphism $$\psi:H^{\,*}_{DR}(\fM)\overset{\simeq}{\lra}H^*(\fM;\C)=H^*(\fM;\Q)\otimes_\Q\C$$ induces a pure Hodge structure on the cohomology of $\fM$. A sufficient condition for the Hodge-to-de-Rham spectral sequence of $\fM$ to degenerate at the $E_1$ page is given in \cite[Theorem~1.4.3]{Kubrak_Prikhodko}. Section 3 of \textit{loc.\ cit.}\ shows that this condition is satisfied for the moduli stack of principal $G$-bundles on a complex smooth projective curve $X$, as well as for its Harder-Narasimhan strata (the definition of the latter will be recalled in Section \ref{HN_stratif_section} of the present paper). Before we get to that, let us collect a few facts on the Hodge structure of $H^*(BG;\C)$.

\smallskip 

Let $H\simeq(\bC^*)^r$ be a maximal  torus in $G$, where $r$ is the rank of $G$. Then  $BH \simeq (\CP^\infty)^r$ as ind-schemes, so $H^*(BH;\bC)$ is a polynomial ring
\begin{equation}\label{classif_stack_max_torus}
H^*(BH;\bC) \simeq \bC[ x_1,\,\ldots\,,x_r]
\end{equation}
where $x_i \in H^2(BH;\bZ)\cap H^{1,1}(BH)$.  Therefore, the Hodge-Poincar\'{e} series of $BH$ is the infinite series
$$
HP_{u,v}(BH) = \frac{1}{(1-uv)^r}\, ,
$$
which is an element in $\bZ[\![u,v]\!] \cap \bZ(u,v)$. The map $BH\lra BG$ (which is a fiber bundle with fiber $G/H$)  induces an injective $\C$-algebra homomorphism compatible with the Hodge decomposition:
$$
H^*(BG;\bC) \lra H^*(BH;\bC) \simeq \bC[x_1,\,\ldots\,,x_r] 
$$
whose image is the sub-algebra fixed by the Weyl group $W:=N(H)/H$, where $N(H)$ is the normalizer of $H$ in $G$.   Therefore, 
$$
H^*(BG;\bC) \simeq \bC[x_1,\,\ldots\,,x_r]^W \simeq \bC[I_1,\,\ldots\,,I_r]
$$
where $I_1,\,\ldots\,, I_r$ are homogeneous polynomials of degree $d_k$ in $(x_1,\,\ldots\,,x_r)$. We have $I_k\in H^{d_k,d_k}(BG)$ for some positive integer $d_k$. By permuting $I_1,\,\ldots\,,I_r$ if necessary, we may assume that
\begin{equation}\label{group_exponents}
d_1 =\cdots = d_m = 1 < d_{m+1}\leq\cdots \leq d_r\,,
\end{equation} 
where $m$ is equal to the dimension of $Z_G$, the center of $G$. So $m=0$ if and only of $G$ is semi-simple. The Hodge-Poincar\'{e} series of $BG$ is
\begin{equation}
HP_{u,v}(BG) = \frac{1}{\prod_{k=1}^r(1- u^{d_k} v^{d_k}) }  =\frac{1}{(1-uv)^m \prod_{k=m+1}^r (1-u^{d_k}v^{d_k})}\,.
\end{equation}
\begin{example}[Classical groups] $ $
\begin{enumerate}
\item[(A)] $G=\GL_r(\bC)$. Then $I_k=c_k$ is the $k$-th elementary symmetric polynomial in the $r$ variables $(x_1,\,\ldots\,,x_r)$ and $d_k=\deg I_k=k$, so
$$
H^*\big(B\GL_r(\bC);\bC\big) \simeq\bC[c_1,\,\ldots\,,c_r]
$$
where $c_k\in H^{k,k}(B\GL_r(\bC))\cap H^{2k}(B\GL_r(\bC);\bZ)$ is the $k$-th Chern class of the universal bundle $E\GL_r(\C)$. Therefore,
$$
HP_{u,v}\big(B\GL_r(\bC)\big) = \frac{1}{\prod_{k=1}^r (1-u^k v^k)}\ . 
$$
The injective group homomorphism $\SL_r(\bC)\lra \GL_r(\bC)$ induces a map $B\SL_r(\bC)\lra B\GL_r(\bC)$ with fiber $\GL_r(\C)/\SL_r(\C)\simeq \C^*$, and a surjective ring homomorphism
$$
H^*\big(B\GL_r(\bC);\bC\big) \simeq \bC[c_1,\,\ldots\,,c_r]\longrightarrow H^*\big(B\SL_r(\bC);\bC\big) \simeq\bC[c_2,\,\ldots\,,c_r]
$$
sending $c_1$ to zero. So:
$$
HP_{u,v}\big(B\SL_r(\bC)\big) =\frac{1}{\prod_{k=2}^r (1-u^k v^k)}\ .
$$

\item[(B)]  $G=\SO_{2r+1}(\bC)$. Then $I_k = p_k$ is the $k$-th universal Pontryagin class, and $d_k = 2k$.
$$
H^*\big(B\SO_{2r+1}(\bC);\bC\big) =\bC[p_1,\,\ldots\,,p_r]
$$
where $p_k\in H^{2k,2k}(B\SO_{2r+1}(\bC))$ is the $k$-th Pontryagin class of $E\SO_{2r+1}(\C)$. Therefore,
$$
HP_{u,v}\big(B\SO_{2r+1}(\bC)\big)  = \frac{1}{\prod_{k=1}^r (1-u^{2k} v^{2k})}\ . 
$$
\item[(C)] $G=\Sp_r(\bC)$. Then $d_k =2k$, and
$$
H^*\big(B\Sp_r(\bC);\bC\big) =\bC[\sigma_1,\,\ldots\,,\sigma_r]
$$
where $\sigma_k \in H^{k,k}(B\Sp_r(\bC))\cap H^{2k}(B\Sp_r(\bC);\bZ)$. Therefore,
$$
HP_{u,v}\big(B\Sp_r(\bC)\big)= \frac{1}{\prod_{k=1}^r (1-u^{2k} v^{2k})}\ .
$$
\item[(D)] $G=\SO_{2r}(\bC)$. Then
$$
H^*\big(B\SO_{2r}(\bC);\bC\big) =\bC[p_1,\,\ldots\,,p_{r-1},e]
$$
where $p_k\in H^{2k,2k}(B\SO_{2n}(\bC))$ is the universal $k$-th Pontryagin class and $e\in H^{r,r}(B\SO_{2r}(\bC))$ is the universal Euler class. Therefore, 
$$
HP_{u,v}\big(B\SO_{2r}(\bC)\big)  = \frac{1}{(1-u^r v^r)\prod_{k=1}^{r-1} (1-u^{2k} v^{2k})}\ . 
$$ 
\end{enumerate}
\end{example}
Over base fields other than $\C$, B.~Totaro has computed the Hodge and de Rham cohomology of the classifying space $BG$ (defined as the \'{e}tale cohomology of the algebraic stack $BG$) in many cases, including for fields of small characteristic \cite{Totaro}. Totaro's definition of the de Rham cohomology of an algebraic stack agrees with the definition of \cite{Behrend, Teleman} for smooth algebraic stacks defined over a field of characteristic zero. 

\subsection{Moduli stacks of principal bundles}\label{stacks_of_G_bdles}
Let $\Sch$ denote the category of $\bC$-schemes, i.e.\ schemes defined over $\Spec\, \bC$. Let $X$ be a smooth projective curve over $\C$ and let $\BunGX$ be the category whose objects are pairs $(S, E)$ where
$S$ is a $\bC$-scheme and $E\lra S\times X$ is a principal $G$-bundle over $S\times X$. A morphism $(S_1,E_1) \lra (S_2, E_2)$ is a cartesian diagram
$$
\xymatrix{
E_1 \ar[r]\ar[d] & E_2 \ar[d] \\
S_1 \times X \ar[r]^{\phi\times id_X}  & S_2 \times X }
$$
where $\phi:S_1\lra S_2$ is a morphism of $\bC$-schemes and $id_X:X\lra X$ is the identity map. The functor
$\BunGX\lra \Sch$, given by $(S,E)\lmt S$, makes $\BunGX$ a category fibered in groupoids over $\Sch$. 
Given a $\bC$-scheme $S$, the fiber $\BunGX(S)$ is the category whose objects are principal $G$-bundles $E$ on $S\times X$ and morphisms are cartesian diagrams
$$
\xymatrix{
E_1 \ar[r]\ar[d] & E_2 \ar[d] \\
S \times X \ar[r]^{id_{S\times X}} & S \times X }
$$
or equivalently, isomorphisms of principal $G$-bundles over $S\times X$. Note that $\BunGX(S)$ is a groupoid since all the morphisms are isomorphisms. We may also view $\BunGX$ as a contravariant $2$-functor
from the category of $\bC$-schemes to the category of groupoids, which is actually a sheaf of groupoids for the étale topology on $\Sch$:
$$
\BunGX: \begin{array}{rcl}
\Sch^{\mathrm{op}} & \lra & \Gp\\
S & \lmt & \BunGX(S)
\end{array}.
$$

\smallskip

For all $\C$-point $s:\Spec\,\bC\lra S$, the pullback of a principal $G$-bundle $E\lra S\times X$ is a principal $G$-bundle $E_s\lra X$ over $X$, i.e.\ an object in $E\in\BunGX(\C)$. 

Recall that, since $\dim_\R X=2$, the topological type of a principal $G$-bundle $E$ over $X$ is characterized by the obstruction classes
$$
o_1(E)\in H^1(X;\pi_0G)\ \mathrm{and}\ o_2(E) \in H^2(X,\piG) \simeq \piG.
$$
Note that $o_1(E)=0$ if $G$ is connected. In this case we call $o_2(E)$ the {\em degree} of $E$. 
\begin{example} 
If $G=\O_{r}(\bC)$, then $\pi_0G\simeq\bZ/2\bZ$ and $o_1(E)=w_1(E) \in H^1(X;\bZ/2\bZ)\simeq(\Z/2\Z)^{2g}$ is the first Stiefel-Whitney class. 
If $G=\GL_r(\bC)$, then $\pi_0(G)=0$, $\pi_1(G)\simeq \bZ$, and $o_2(E)  = c_1(E) \in H^2(X,\bZ)$ is the first Chern class. 
If $G=\SO_{2r}(\bC)$, then $\pi_0 G =0$,  $\piG\simeq\bZ/2\bZ$, and $o_2(E)=w_2(E) \in H^2(X;\bZ/2\bZ)\simeq\Z/2\Z$ is the second Stiefel-Whitney class. 
\end{example}

The above topological classification holds when $G$ is a topological group and $X$ is a connected closed oriented topological 2-manifold. The isomorphism
$H^2(X,\piG) \simeq \piG$ depends on the  orientation on $X$.  When $G$ is 
a connected topological group, so that $o_1(E)=0$,  a proof of this topological classification can be found in \cite[Section 5]{Ramanathan}, where $o_2(E)$ is denoted $\chi(E)$. When $G$ is a Lie group and $X$ is a connected closed Riemann surface, the topological and $C^\infty$ classification of $G$-bundles over $X$ are the same. 

We can write $\BunGX$ as a disjoint union of connected components (i.e.\ irreducible open and closed substacks):
$$
\BunGX = \bigsqcup_{d\in \piG} \BunGXd
$$
where $\BunGXd$ is the moduli stack of principal $G$-bundles of degree $d$ over $X$, meaning that, for all $\C$-scheme $S$, the category $\BunGXd(S)$ is the groupoid 
$$
\BunGXd(S) = \left<E\in \BunGX(S)\ |\ \forall\,s\in S(\C), \deg\, E_s=d\right>.
$$  

By the Yoneda lemma for stacks, for all $\C$-scheme $T$, there is a natural equivalence of categories
$$
\Mor_{\text{stacks}}(\underline{T},\BunGXd) \cong \BunGXd(T)
$$
where $\underline{T}: \Sch^{\text{op}} \to \text{Sets} \subset \text{Groupoids}$ sends
a $\bC$-scheme $S$ to the set  $\underline{T}(S):= \Mor_{\Sch}(S,T)$. 
The universal principal $G$-bundle  $U\to \BunGXd\times X$  is the principal $G$-bundle associated to the identity morphism 
$$
\mathrm{id}_{\BunGXd}\in \Mor_{\text{stacks}} \left( \BunGXd, \BunGXd\right).
$$

Let $\fg$ be the Lie algebra of $G$. Given a principal $G$-bundle $E$ over a scheme or an algebraic stack $B$, let $\ad\, E = E\times_{\Ad} \fg \to B$ be the 
vector bundle associated to the adjoint representation $\Ad: G\to \GL(\fg)$.  The
tangent complex of the smooth algebraic stack $\BunGXd$ is given by 
$$
\cT_{\BunGXd} = \mathbf{R}\pi_* \ad\, U[-1]
$$
where $\pi: \BunGXd\times X\to \BunGXd$ is the projection onto the first factor. The tangent bundle of $\BunGXd$ is
$$
T_{\BunGXd}= h^1/h^0\left(\mathbf{R}\pi_*\ad\, U\right)
$$
which is a vector bundle stack over $\BunGXd$. Given an object $E\to X$ in $\BunGXd(\bC)$, or equivalently a morphism $\xi: \Spec(\bC) \to \BunGXd$, the
fiber of $T_{\BunGXd}$ over $\xi$ is the quotient stack
\begin{equation}\label{infinitesimal_deformations}
[ H^1(X, \ad\, E)/H^0(X, \ad\, E)]
\end{equation}
where $H^0(X,\ad\, E)$ (resp. $H^1(X,\ad\, E)$) is the space of infinitesimal automorphisms (resp.\ deformations) of the principal $G$-bundle $E\to  X$, and $H^0(X, \ad\, E)$ acts trivially on $H^1(X,\ad\, E)$. 

\smallskip

The dimension of the smooth algebraic stack $\BunGXd$ is equal to the rank of its tangent bundle, which, by Riemann-Roch on $X$, is equal to 
$$
\dim H^1(X,\ad\, E) - \dim H^0(X,\ad\, E) =  - \left( \deg(\ad\, E) + \rank(\ad\, E) (1-g) \right) = \dim G (g-1)
$$
since $\rank(\ad\, E) =\dim G$ and $\deg(\ad\, E)=0$, the latter because $\ad\,E$ is self-dual when $G$ is reductive.

\smallskip

\begin{remark}\label{differential_viewpoint}
We refer for instance to \cite{AB,Ho_Liu_YM_AMS} for more details on the following differential-geometric viewpoint. Take $d\in\piG$ and let $\cP$ be a $C^\infty$ principal $G$-bundle of degree $d$ over $X$. Let $\cC$ be the infinite-dimensional complex affine space of $(0,1)$-connections on $\cP$, and let $\cG$ be the gauge group of $\mathscr{P}$. The latter is an infinite-dimensional Lie group, and we denote by $B\cG$ its classifying space in the homotopy-theorectic sense. For $G=\GL_r(\C)$, Atiyah and Bott proved in \cite{AB} that the graded $\Z$-algebra $H^*(B\cG;\Z)$ is a torsion-free $\Z$-module which is finitely generated in each degree, and computed its Poincaré series. For a general reductive group $G$, the rational Poincaré series of $\BunGXd$ was computed by Laumon and Rapoport \cite{LR}. The point is that, for a reductive complex algebraic group $G$, we have:
\begin{equation}\label{eqn:BunG-BcG} 
 H^*\big(\BunGXd;\bQ\big) \simeq  H^*_{\cG}\big(\cC;\bQ\big) \simeq  H^*\big(B\cG;\bQ\big). 
 \end{equation}
\end{remark}

In this paper, we will be interested mostly in the cohomology of $\BunGXd$ with coefficients in $\C$, which does not change the Poincaré series since $H^*(\BunGXd;\C)\simeq H^*(\BunGXd;\Q)\otimes_\Q\C$ by the universal coefficient theorem.

 \begin{theorem}[Atiyah-Bott \cite{AB}, Laumon-Rapoport \cite{LR}] \label{thm:stack-Poincare} Let $m$ and $(d_k)_{1\leq k\leq r}$ be defined as in \eqref{group_exponents}. Then:
 $$
P_t\big(\BunGXd;\bQ\big)=   \left(\frac{(1+t)^{2g}}{1-t^2} \right)^m \prod_{k=m+1}^r \frac{(1+t^{2d_k-1})^{2g}}{(1-t^{2d_k-2})(1-t^{2d_k})}. 
 $$
 \end{theorem} 
 \begin{corollary}
If $G=\GL_r(\bC)$, then $m=1$ and $d_k=k$, so:
$$
P_t\big(\BunGXd;\bQ\big) =  \frac{(1+t)^{2g}}{1-t^2}  \prod_{k=2}^r \frac{(1+t^{2k-1})^{2g}}{(1-t^{2k-2})(1-t^{2k})}. 
$$
\end{corollary}
 
\subsection{Hodge theory of moduli stacks of bundles}

\subsubsection{The Hodge-Poincar\'{e} series} 

In \cite[Section 4]{Teleman}, C.~Teleman proved the existence of a pure Hodge structure on $H^*(\BunGXd;\bC)$ and computed its Hodge-Poincaré series in the case when $G$ is semisimple \cite[Proposition~4.4]{Teleman}. His formula was also obtained by K.~Behrend and A.~Dhillon in \cite[Remark~4.2]{BD}, as a consequence of the main conjecture in that paper (see also \cite{dARN}). Here, we extend Teleman's result to the reductive case, which constitutes a refinement of Theorem \ref{thm:stack-Poincare}. 

\begin{theorem}\label{thm:Hodge}
Given a reductive complex algebraic group $G$ , the Hodge-Poincaré series of the moduli stack of principal $G$-bundles of degree $d$ over a complex smooth proper curve $X$ is given by the formula:
\begin{equation}\label{HP_series_of_BunGXd}
HP_{u,v}\big(\BunGXd\big) =  \left(\frac{ (1+u)^g(1+v)^g }{1-uv} \right)^m \prod_{k=m+1}^r
 \frac{(1+u^{d_k} v^{d_k-1})^g (1+u^{d_k-1} v^{d_k})^g}{(1- u^{d_k-1} v^{d_k-1})(1-u^{d_k} v^{d_k})}\,,
 \end{equation}
where $g$ is the genus of $X$ and $m$ and $(d_k)_{1\leq k\leq r}$ are defined as in \eqref{group_exponents}. 
 \end{theorem}

\begin{remark}
Note that the right hand side of \eqref{HP_series_of_BunGXd} does not depend on $d$, so we will denote it simply by $a_{u,v}(G_X)$. More generally, we introduce the following notation, which will be useful in Section \ref{inversion_section}. For all reductive subgroup $L\subset G$, we denote by $m(L)$ the complex dimension of $Z_L$, by $r(L)$ the rank of $L$ and, as in \eqref{group_exponents}, by $(d_i(L))_{1\leq i \leq r(L)}$ the degrees of the homogeneous polynomials that generate $H^*(BL;\C)$. Theorem \ref{thm:Hodge} then shows that, for all $\ud\in\pi_1L$, the Hodge-Poincaré series of the moduli stack of principal $L$-bundles of topological type $\ud$ on $X$ is given by the following formula:
\begin{equation}\label{HP_series_of_BunLXd}
a_{u,v}(L_X):= HP_{u,v}(\fBun_{L_X}^{\,\delta}) =  \left(\frac{ (1+u)^g(1+v)^g }{1-uv} \right)^{m(L)} \prod_{k=m(L)+1}^{r(L)}
 \frac{(1+u^{d_k(L)} v^{d_k(L)-1})^g (1+u^{d_k(L)-1} v^{d_k(L)})^g}{(1- u^{d_k(L)-1} v^{d_k(L)-1})(1-u^{d_k(L)} v^{d_k(L)})}\,.
\end{equation}
\end{remark}

When $G$ is semi-simple, or equivalently when $m=0$, Formula \eqref{HP_series_of_BunGXd} indeed coincides with \cite[Remark 4.2]{BD}. When $G=\GL_r(\bC)$, we get the following result, which appears under a different form in \cite{Earl_Kirwan}.

 \begin{corollary}\label{HP_stack_of_vector_bundles}
 If $G=\GL_r(\bC)$, then $m=1$ and $d_k=k$, so:
 $$
HP_{u,v}\big( \BunGXd\big) =  \frac{(1+u)^g(1+v)^g} {1-uv}  \prod_{k=2 }^r
 \frac{(1+u^k v^{k-1})^g (1+u^{k-1} v^k)^g}{(1- u^{k-1} v^{k-1})(1-u^k v^k)}\,.
 $$
 \end{corollary} 
 
 The proof of Theorem \ref{thm:Hodge} will be given in Section \ref{proof_thm_Hodge}.
 
 \subsubsection{K\"{u}nneth decompositions of universal classes}\label{proof_thm_Hodge}

For the proof of Theorem \ref{thm:Hodge}, we need to recall a few facts about the Hodge theory of a complex smooth projective curve and the Atiyah-Bott proof of Theorem \ref{thm:stack-Poincare}.
 
\smallskip

Let $X$ be an irreducible smooth projective curve of genus $g$ over $\bC$. In this subsection, we view it as a compact connected Riemann surface of genus $g$. The underlying real manifold has a canonical orientation, induced by the complex structure of $X$. A Torelli structure on $X$ is a choice of an integral symplectic basis  $(\alpha_j, \beta_j)_{1\leq j\leq g}$  of $(H_1(X;\bZ),\cap)$, where $\cap$ is the intersection pairing.
Let $(\alpha_i^*, \beta_j^*)_{1\leq j\leq g}$ be the dual integral basis of $H^1(X;\bZ)$, so that
$$
\int_{\alpha_i}\alpha_j^* = \int_{\beta_i}\beta_j^*=\delta_{ij}, \quad
\int_{\alpha_i}\beta_j^* = \int_{\beta_i}\alpha_j^*=0. 
$$

\noindent Let $\omega_1,\,\ldots\,,\omega_g$ be holomorphic 1-forms on $X$ such that
\begin{equation}\label{def_omega_j}
\int_{\alpha_i}\omega_j = \delta_{ij}.
\end{equation}
Then 
$$
H^{1,0}(X) =\bigoplus_{j=1}^g \bC \omega_j,\quad H^{0,1}(X) =\bigoplus_{j=1}^g \bC\bar{\omega}_j. 
$$
The period matrix of $\left(X, (\alpha_j,\beta_j)_{1\leq j\leq g}\right)$ is $\tau=(\tau_{ij})$, where
\begin{equation}\label{curveVHS}
	\tau_{ij}= \int_{\beta_i}\omega_j. 
\end{equation}
By Riemann's bilinear relations, $\tau$ is a complex $g\times g$ matrix in the Siegel upper half-space $\cH_g$, where
$$
\cH_g= \left\{ \tau=(\tau_{ij}) \in M_{g\times g}(\bC) \mid \tau_{ij}=\tau_{ji},\;  \Im(\tau)  \textup{ is positive definite } \right\}.
$$
The two bases
$$
(\omega_j, \bar{\omega}_j)_{1\leq j\leq g}, \quad (\alpha_j^*, \beta_j^*)_{1\leq j\leq g}
$$ 
of $H^1(X;\bC)$ are related as follows:
$$
\omega_i = \alpha_i^*  +\sum_{j=1}^g \tau_{ij}\beta_j^*, \quad
\bar{\omega}_i = \alpha_i^* +\sum_{j=1}^g \bar{\tau}_{ij} \beta_j^*.   
$$
\begin{eqnarray*}
\alpha_i^*  &=&  \mathrm{Re}\, \omega_i  - \sum_{j=1}^g \left( \mathrm{Re}(\tau)\,\Im(\tau)^{-1}\right)_{ij} \Im \omega_j \\
&=& \frac{1}{2} \sum_{i=1}^g \left(\delta_{ij}+ \sqrt{-1} \left( \mathrm{Re}(\tau)\,\Im(\tau)^{-1}\right)_{ij}\right)\omega_j +
 \frac{1}{2} \sum_{i=1}^g \left(\delta_{ij} - \sqrt{-1}  \left(\mathrm{Re}(\tau)\,\Im(\tau)^{-1}\right)_{ij}\right)\bar{\omega}_j 
\end{eqnarray*} and  
\begin{eqnarray*}
\beta^*_i &=& \sum_{j=1}^g  \big(\Im(\tau)^{-1}\big)_{ij}\, \Im\omega_j \\
&=&  \frac{-\sqrt{-1}}{2} \sum_{j=1}^g  \big(\Im(\tau)^{-1}\big)_{ij}\, \omega_j
+   \frac{\sqrt{-1}}{2} \sum_{j=1}^g  \big(\Im(\tau)^{-1}\big)_{ij}\,  \bar{\omega}_j\,.
\end{eqnarray*}

\smallskip 

We now recall the definition of the generators of $H^*(\BunGXd;\bQ)$ constructed by Atiyah-Bott  \cite{AB}, known as the Atiyah-Bott classes (see also \cite[Section 3.1]{Heinloth_Schmitt}). For $k=1,\,\ldots\,,r$, 
consider the K\"{u}nneth decomposition of  the characteristic class $I_k(U) \in H^{2d_k}(\BunGXd\times X;\bQ)$ of the universal principal $G$-bundle $U\lra \BunGXd\times X$:
\begin{equation}\label{eqn:Kunneth} 
I_k(U) = h _k\otimes 1 +\sum_{j=1}^g  a^j_k \otimes \alpha_j^* +\sum_{j=1}^g b^j_k \otimes \beta_j^* +  f_k \otimes \omega
\end{equation}
where $\omega\in H^2(X;\bZ)$ is the Poincar\'{e} dual of the class of a point $[x_0] \in H_0(X;\bZ)$, and 
\begin{equation}\label{degree_generators}
h_k \in H^{2d_k}(\BunGXd;\bQ),\quad
a^j_k, b^j_k \in H^{2d_k-1}(\BunGXd;\bQ),\quad
f_k \in H^{2d_k-2}(\BunGXd;\bQ). 
\end{equation} 
Atiyah and Bott showed that 
\begin{equation}\label{eqn:algebra}
H^*(\BunGXd;\bQ) =\bQ[ h_1,\,\ldots\,,h_r, f_{m+1},\,\ldots\,, f_r] \otimes_{\bQ} \Lambda_{\bQ}[ (a^j_k, b^j_k)_{1\leq j\leq g, 1\leq k\leq r}]. 
\end{equation} 
Theorem \ref{thm:stack-Poincare} then follows from \eqref{degree_generators} and \eqref{eqn:algebra}.

 \begin{proof}[Proof of Theorem \ref{thm:Hodge}]
The universal bundle $U\lra \BunGXd \times X$ is algebraic, so 
$$
I_k(U)\in H^{d_k,d_k}(\BunGXd\times X). 
$$
Let $(\omega_j,\bar{\omega}_j )_{1\leq j\leq g}$ be defined as in \eqref{def_omega_j}. Equation \eqref{eqn:Kunneth} can be rewritten as
\begin{equation}\label{eqn:Kunneth-C} 
I_k(U) = h _k\otimes 1 +\sum_{j=1}^g  \theta^j_k \otimes \omega_j  +\sum_{j=1}^g \bar{\theta}^j_k \otimes \bar{\omega}_j +  f_k \otimes \omega
\end{equation}
where
\begin{equation}\label{eqn:theta-a-b}
\theta^j_k = \frac{1}{2} \sum_{i=1}^g \left(\delta_{ij} + \sqrt{-1} \left( \mathrm{Re}(\tau)\Im(\tau)^{-1}\right)_{ij}\right) a^i_k
-\frac{\sqrt{-1}}{2}\sum_{i=1}^g\left(\Im(\tau)^{-1}\right) b^i_k 
\end{equation}
and $\bar{\theta}^j_k$ is the complex conjugate of $\theta^j_k$. 
The K\"{u}nneth decomposition is compatible with the Hodge decomposition (\cite[Proposition~8.2.10]{Deligne_Hodge_III}) and
$$
1\in H^{0,0}(X),\quad \omega_j \in H^{1,0}(X),\quad \bar{\omega}_j\in H^{0,1}(X),\quad \omega\in H^{1,1}(X).
$$
Therefore, 
\begin{equation}\label{eqn:bidegree} 
\begin{aligned}
& h_k \in H^{d_k, d_k}  \left( \BunGXd \right),\quad
\theta^j_k \in H^{d_k-1, d_k }\left( \BunGXd\right),\\
& \bar{\theta}^j_k\in H^{d_k,d_k-1}\left( \BunGXd\right),\quad
f_k \in H^{d_k-1,d_k-1}\left( \BunGXd\right).
\end{aligned} 
\end{equation} 
Equations \eqref{eqn:algebra} and \eqref{eqn:theta-a-b} imply that
\begin{equation}\label{eqn:algebra-C}
H^*(\BunGXd;\bC) =\bC[ h_1,\,\ldots\,,h_r, f_{m+1},\,\ldots\,, f_r] \otimes_{\bC} \Lambda_{\bC}[ (\theta^j_k, \bar{\theta}^j_k)_{1\leq j\leq g, 1\leq k\leq r}]. 
\end{equation} 
Theorem \ref{thm:Hodge} then follows from \eqref{eqn:bidegree} and \eqref{eqn:algebra-C}. 
 \end{proof} 

In \cite{Teleman}, Teleman observes that the Atiyah-Bott isomorphism of $\C$-algebras $$H^*\big(\BunGXd;\C\big) \simeq H^*\big(\Omega(\widetilde{G_{ss}});\C\big) \otimes H^*\big(G;\C\big)^{\otimes 2g} \otimes H^*\big(B G;\C\big)$$ refines to a multiplicative splitting of pure $\Q$-Hodge structures. He points out that the Hodge structures on the cohomology of $\Omega(\widetilde{G_{ss}})$ and $BG$ are Hodge-Tate, while that of $H^*(G;\C)^{\otimes 2g}$ is not: as the proof of Theorem \ref{thm:Hodge} shows, this reflects the fact that the (pure) Hodge structure of the base curve $X$ is not Hodge-Tate, which, in view of the Künneth decomposition \eqref{eqn:Kunneth-C}, causes a generator of $H^*(G;\C)$ of degree $2k-1$ to induce generators of degree $(k,k-1)$ and $(k-1,k)$ in $H^*(\BunGXd;\C)$ \cite[Proposition~4.4]{Teleman}.

The complex algebraic structure on $\BunGXd$ depends on the complex algebraic structure on $X$. In the course of the proof of Theorem \ref{thm:Hodge}, we also computed the variation of Hodge structures on $H^*(\BunGXd;\bC)$ in terms of the variation of Hodge structure on $H^*(X;\bC)$, which is governed by the period matrix $\tau$ of $X$ defined as in \eqref{curveVHS}.  

\begin{corollary}\label{cor:VHS}
	Let $\tau$ be the period matrix of $X$ defined in \eqref{curveVHS}. Then the variation of Hodge structure on $\BunGXd$ is given by the Hodge classes
	\begin{equation}
		\theta^j_k = \frac{1}{2} \sum_{i=1}^g \left(\delta_{ij} + \sqrt{-1} \left( \mathrm{Re}(\tau)\Im(\tau)^{-1}\right)_{ij}\right) a^i_k
		-\frac{\sqrt{-1}}{2}\sum_{i=1}^g\left(\Im(\tau)^{-1}\right) b^i_k \in H^{d_k - 1, d_k}(\BunGXd),
	\end{equation}
	for $k \in \{1,\ \ldots\ , r\}$ and $j \in \{1,\ \ldots\ , g\}$.
\end{corollary}

\begin{proof}
	This is exactly Equation \eqref{eqn:theta-a-b}.
\end{proof}
 
\section{Recursive formulae for semistable strata}\label{recursive_formulae_section}

\subsection{Preliminaries on connected complex reductive Lie groups and their parabolic subgroups}
Let $G$ be a connected complex reductive Lie group. We fix once and for all a maximal torus $H\subset G$ with Lie algebra $\fh\simeq\C^r$, where $r$ is the rank of $G$, and a Borel subgroup $B\supset H$. The choice of $H$ determines a root system $R\subset \fh^\vee$, and the choice of $B$ is equivalent to that of a set of positive roots $R_+\subset R$. This in turn determines a set of simple roots $\Delta\subset R_+$. We shall denote by 
$$
\fg = \fh \oplus \bigoplus_{\alpha\in R} \fg_\alpha
$$ 
the root space decomposition of $\fg$ and by $$\fb = \fh \oplus \bigoplus_{\alpha\in R_+} \fg_\alpha$$ the Lie algebra of $B$.

\smallskip

A parabolic subgroup $P\subset G$ is a subgroup that contains a Borel subgroup. It is called \textit{standard with respect to $B$} if $P\supset B$, and standard parabolic subgroups correspond bijectively to subsets $I\subset \Delta$. Indeed, a standard parabolic subgroup $P\supset B$ is entirely determined by its Lie algebra $\fp \supset \fb$, and the latter is of the form 
$$
\fp = \fh\oplus \bigoplus_{\alpha\in\Gamma} \fg_\alpha$$ where 
$$
\Gamma = R_+ \cup \{\alpha\in R\ |\ \alpha\in\mathrm{span}_\Z(\Delta\setminus I)\}.
$$ 
Define
$$
-\Gamma := \{ -\alpha \mid \alpha \in \Gamma\} = -R_+ \cup  \{ \alpha \in R \mid \alpha \in \mathrm{span}_\Z(\Delta\setminus I)\},
$$
and let
$$
\fl :=\fh \oplus\bigoplus_{\alpha \in \Gamma\cap -\Gamma} \fg_\alpha, \quad \fu:= \bigoplus_{\alpha \in \Gamma, \alpha\notin -\Gamma} \fg_\alpha.
$$
Then
\begin{enumerate}
\item[($\fp1$)] $\fl$, $\fu$ are Lie subalgebras of $\fp$: $[\fl, \fl]\subset \fl$ and $[\fu, \fu]\subset \fu$.
\item [($\fp 2$)] $\fu$ is an ideal of $\fp$: $[\fp, \fu] \subset \fu$. 
\item [($\fp 3$)] $\fp = \fl\oplus \fu$. 
\item [($\fp 4$)] The subalgebra $\fl$ is reductive, and is called the {\em Levi factor} of $\fp$.
\item [($\fp 5$)] The subalgebra $\fu$ is nilpotent, and is called the {\em nilpotent radical} of $\fp$.
\end{enumerate} 
Moreover,  we have the following statements  (P1)-(P5) on Lie subgroups of $P$ which correspond to the above
($\fp 1$)-($\fp 5$) on Lie subalgebras of $\fp$. 
\begin{enumerate}
\item[(P1)] There is a unique closed subgroup $L$ (resp. $U$) of $P$ with Lie algebra  $\fl$ (resp. $\fu$).
\item[(P2)] $U$ is a normal subgroup of $P$, so that $P/U$ is a group. 
\item [(P3)]  The composition of the inclusion  $L\to P$ and the quotient $P\to P/U$ is a group isomorphism: $L\simeq P/U$. Equivalently, 
$P$ is a semi-direct product of $L$ and $U$:  $P = U \rtimes L$.  
\item[(P4)] The subgroup $L$ is reductive, and is called the {\em Levi factor} of $P$.
\item[(P5)] The subgroup $U$ is unipotent, and is called the {\em unipotent radical} of $P$. 
\end{enumerate} 
Let
$$
\fu_- := \bigoplus_{\alpha \in \Gamma, \alpha\notin -\Gamma} \fg_{-\alpha} = \bigoplus_{\alpha  \in -\Gamma, \alpha \notin \Gamma} \fg_\alpha
$$
Then $\fp_- := \fl \oplus \fu_-$ is the parabolic subalgebra of $\fg$ determined by the system of simple roots $-\Delta$ and the subset $-I \subset -\Delta$,
and  $\fl$ (resp. $\fu_-$) is the Levi factor (resp.  nilpotent radical) of $\fp_-$.  We have
\begin{equation} \label{eqn:luu} 
\fg = \fl \oplus \fu \oplus \fu_- = \fp \oplus \fu_-.
\end{equation} 
In particular, $\fg/\fp = \fu_-$. 

\begin{remark}
	We will see after Definition \ref{def_HN_type} that a decomposition of the form \eqref{eqn:luu} is determined by a so-called \textit{Atiyah-Bott point} $\mu \in \overline{C_0} \subset \fh_\bR$, in the following sense:
	\begin{eqnarray*}
		\fl &=& \fh \oplus \bigoplus_{\alpha \in R,  \alpha(\mu)=0} \fg_\alpha, \\
		\fu &=& \bigoplus_{\alpha \in R,  \alpha(\mu)>0} \fg_\alpha = \bigoplus_{\alpha \in R_+, \alpha(\mu)>0} \fg_\alpha, \\
		\fu_- &=&  \bigoplus_{\alpha \in R, \alpha(\mu)<0} \fg_\alpha = \sum_{\alpha \in R_+, \alpha(\mu)>0} \fg_{-\alpha} 
	\end{eqnarray*} 
	Therefore, $\fp=\fl\oplus \fu$ and $I \subset \Delta$ are also determined by $\mu$; indeed $I =\{ \alpha \in \Delta\mid \alpha(\mu)>0\}$. 
\end{remark}

Let $Z_G$ be the center of $G$, and let $(Z_G)_0$ denotes its connected component of the identity. 
Then $(Z_G)_0\simeq  (\bC^*)^m$, where $m$ is defined as in Equation \eqref{group_exponents}.
Let $\fg$, $\fh$, and $\fz_G$ denote the Lie algebras of $G$, $H$,  and $Z_G$, respectively.  We have
$$
\fh = \fz_G \oplus \fh' 
$$ 
where $\fh' = \fh \cap [\fg,\fg]\simeq \bC^{r-m}$. 
We fix a maximal compact subgroup $K\subset G$ and denote by $T$ the compact maximal torus $H\cap K$ in $K$. 
Following \cite{FM}, we define\footnote{The definition of $\fh_\bR$ in this paper agrees with that in \cite{FM} and corresponds to 
$\sqrt{-1}\fh_\bR$ in \cite{Ho_Liu_YM_AMS} (see e.g. \cite[Section 3]{Ho_Liu_YM_AMS}).} 
$$
\fg_\bR :=\sqrt{-1} \Lie(K), \: \fh_\bR :=\fh\cap \fg_\bR,\; \fh'_\bR :=\fh'\cap \fg_\bR, \:  (\fz_G)_\bR := \fz_G \cap \fg_{\bR}.
$$ 
Then $\Lie(K) =\sqrt{-1}\fg_\bR$ and $\Lie(T) =\sqrt{-1}\fh_\bR$. 
Note that $T\cong \U(1)^r$, $\fh_\bR \simeq \bR^r$, $\fh'_\bR\simeq \bR^{r-m}$, and $(\fz_G)_\bR \simeq \bR^m$.  
Let $\fh_\bR^*$ and $(\fz_G)_\bR^*$ be the dual real vector spaces of $\fh_\bR$ and $(\fz_G)_\bR$, respectively. 
Let $[G,G]$ be  {\em derived subgroup} (also known 
the {\em commutator subgroup})  of $G$.  Then $[G,G]$ is a connected semisimple complex reductive Lie group
with $\Lie([G,G])= [\fg,\fg]$, and  $D :=(Z_G)_0\cap [G,G]$ is a finite abelian group. The {\em abelianization} of $G$ is
$$
G^{\ab} := G/[G,G]\cong (Z_G)_0/D \cong (\bC^*)^m. 
$$ 
\begin{enumerate}
\item The {\em coweight lattice} is 
$$
\piH \cong \Hom(\bC^*, H) = \Hom(\U(1), T) \cong \pi_1(T) \cong \{ X\in \fh_\bR \mid \exp(2\pi\sqrt{-1}X) =e \} \subset \fh_\bR,
$$
where $e$ is the identity element of $G$, and $\Hom(\bC^*,H)$ is the cocharacter lattice of the torus $H$. We fix an isomorphism $H\cong  (\bC^*)^r$, which determines a $\bZ$-basis $\{ e_1,\ldots,e_r  \}$ of the lattice $\piH\cong \bZ^r$. We have $\pi_1 H\otimes_{\bZ} \bR =\fh_\bR$.

\item The {\em weight lattice} is the character lattice of $H$:
$$
\widehat{H} := \Hom(H,\bC^*) = \Hom(T, \U(1)).  
$$
It is the lattice in $\fh_\bR^*$ dual to the lattice $\piH \subset \fh_\bR$. Let $\{\theta_1,\ldots, \theta_r\}$ be the $\bZ$-basis
of $\widehat{H}$ dual to the $\bZ$-basis $\{e_1,\ldots,e_r\}$ of the coweight lattice. Then $\langle \theta_i, e_j\rangle =\delta_{ij}$, and  $\widehat{H}\otimes_{\bZ}\bR = \fh_\bR^*$. 

\item   The {\em coroot lattice} of $G$ is the sublattice $\Lambda$ of $\pi_1(H)$ generated by the set $\Delta^\vee$ of simple coroots, and is the kernel
of the (surjective) group homomorphism $\pi_1 H\to \pi_1 G$ induces by the inclusion $H\to G$; therefore, $\piG =\piH/\Lambda$.
We have $\Lambda\otimes_{\bZ}\bR =\fh'_\bR$.  

\item Given $\alpha\in \Delta$, let $\varpi_\alpha$ be the unique element in $\fh^*$ such that
$\langle \varpi_\alpha, \beta \rangle =0$ for $\beta\in \fz_G$ and
$\langle \varpi_\alpha, \beta^\vee\rangle =\delta_{\alpha\beta}$ for all $\beta \in \Delta$.

\item The  {\em character lattice}  of $G$ is $\widehat{G} =\Hom(G,\bC^*) = \Hom (G^\ab, \bC^*)$ which is a lattice
in $(\fz_G)_\bR^*$.   We have $\widehat{G}\otimes_{\bZ}\bR = (\fz_G)_\bR^*$. 

\item Let $\widehat{\Lambda} := (\Lambda \otimes_{\bZ} \bQ)\cap \pi_1 H$ be the saturation of $\Lambda$ in $\pi_1 H$. Then $\pi_1 [G,G] \simeq \widehat{\Lambda}/\Lambda$
is a finite abelian group,   $\pi_1(G^\ab) \simeq \piH/\widehat{\Lambda}$  is a lattice in $(\fz_G)_\bR$ which is the dual of 
the lattice $\widehat{G}$ in $(\fz_G)_\bR$.   We have $\widehat{\Lambda}\otimes_{\bZ}\bR =\fh'_\bR$ and  $\pi_1(G^\ab)\otimes_{\bZ}\bR  = (\fz_G)_\bR$. 
\end{enumerate}

By the above (3) and (5), there is a short exact sequence of multiplicative abelian groups
$$
1\to \pi_1[G,G] \to \piG \to  \pi_1(G^\ab) \to 1
$$
which can be identified with the following short exact sequence of additive groups: 
$$
0 \to \widehat{\Lambda}/\Lambda \to \piH/\Lambda \to \piH/\widehat{\Lambda} \to 0.
$$

\begin{example} \label{GLroots} $G=\GL_r(\bC)$,  $[G,G]=\SL_r(\bC)$.
$$
\piH =\{ \diag(\mu_1,\ldots,\mu_r)\mid \mu_i \in \bZ\} \subset \fh_\bR =\{ \diag(\mu_1,\ldots,\mu_r)\mid \mu_i \in \bR\}.
$$
We choose a $\bZ$-basis $e_1,\ldots, e_r$ of $\piH$ so that    $\diag(\mu_1,\ldots,\mu_r) =\sum_{j=1}^r \mu_j e_j$.  
$$
\begin{aligned}
\Delta  =   &\;  \{ \alpha_i = \theta_i-\theta_{i+1} \mid i=1,\ldots,r-1\} \subset \fh_\bR^*, \\
\Delta^\vee = &\;  \{ \alpha_i^\vee = e_i -e_{i+1} \mid i=1,\ldots, r-1\} \subset \fh_\bR, \\
\Lambda = & \; \Big\{ \diag(\mu_1,\ldots, \mu_r) =\sum_{j=1}^r \mu_j e_j  \; \big| \; \mu_j \in \bZ, \sum_{j=1}^r \mu_j =0 \Big\} =\hLambda,\\
\widehat{G} = & \;  \bZ(\theta_1+\cdots + \theta_r) \subset  (\fz_G)_\bR^* =\bR (\theta_1+\cdots \theta_r),\\
\piG^{\ab} =&\;  \bZ  \frac{ e_1+\cdots + e_r}{r} = \bZ \, \Big( \frac{1}{r} I_r \Big) \subset  (\fz_G)_\bR = \bR I_r,
\end{aligned}
$$
where $I_r$ denotes the $r\times r$ identity matrix. 
$$
\varpi_{\alpha_i} = \theta_1 +\cdots + \theta_i  -\frac{i}{r} \sum_{j=1}^r \theta_j,  \quad i=1,\ldots, r-1. 
$$
In this case $\pi_1[G,G] = \pi_1\SL_r(\bC) =\{1\}$ and $\piG\simeq \piG^\ab = \bZ$. 
\end{example}

\begin{example} \label{SOodd-roots} 
$G=\SO_{2r+1}(\bC)=[G,G]$. Let 
\begin{equation}\label{eqn:J}
J =\begin{bmatrix} 0 &-1\\ 1 & 0 \end{bmatrix}.
\end{equation} 
$$
\piH =\{ \diag(\mu_1 J,\ldots,\mu_rJ, 0 I_1)\mid \mu_i \in \bZ\} \subset \fh_\bR =\{ \diag(\mu_1 J,\ldots,\mu_r J, 0I_1)\mid \mu_i \in \bR\}.
$$
We choose a $\bZ$-basis $e_1,\ldots, e_r$ of $\piH$ so that    $\diag(\mu_1 J,\ldots,\mu_r J, 0I_1) =\sum_{j=1}^r \mu_j e_j$.  
$$
\begin{aligned}
\Delta  =   &\;  \{ \alpha_i = \theta_i-\theta_{i+1} \mid i=1,\ldots,r-1\}  \cup \{ \alpha_r =\theta_r\} \subset \fh_\bR^*, \\
\Delta^\vee = &\;  \{ \alpha_i^\vee = e_i -e_{i+1} \mid i=1,\ldots, r-1\} \cup \{ \alpha_r^\vee = 2e_r \} \subset \fh_\bR, \\
\Lambda = & \; \Big\{ \diag(\mu_1 J,\ldots, \mu_r J,0I_1) =\sum_{j=1}^r \mu_j e_j  \; \big| \; \mu_j \in \bZ, \sum_{j=1}^r \mu_j \in 2\bZ \Big\},\\
\hLambda = &\: \pi_1 H =\bigoplus_{i=1}^r \bZ e_i. 
\end{aligned}
$$
$$
\varpi_{\alpha_i} =  \theta_1 +\cdots + \theta_i ,  \hspace{0.2cm}  i=1,\ldots,r-1; \quad
\varpi_{\alpha_r}= \frac{1}{2}( \theta_1+\cdots+\theta_r) . 
$$
In this case $G^{\ab}$ is trivial and $\pi_1 G =\pi_1 [G,G] = \hLambda/\Lambda = \langle e_r\rangle  \simeq \bZ/2\bZ$. 
 \end{example}

\begin{example} \label{SOeven-roots} 
$G=\SO_{2r}(\bC)=[G,G]$. Let $J$ be as in Equation \eqref{eqn:J}. 
$$
\piH =\{ \diag(\mu_1 J,\ldots,\mu_rJ)\mid \mu_i \in \bZ\} \subset \fh_\bR =\{ \diag(\mu_1 J,\ldots,\mu_r J)\mid \mu_i \in \bR\}.
$$
We choose a $\bZ$-basis $e_1,\ldots, e_r$ of $\piH$ so that    $\diag(\mu_1 J,\ldots,\mu_r J) =\sum_{j=1}^r \mu_j e_j$.  
$$
\begin{aligned}
\Delta  =   &\;  \{ \alpha_i = \theta_i-\theta_{i+1} \mid i=1,\ldots,r-1\}  \cup \{ \alpha_r =\theta_{r-1}+ \theta_r\} \subset \fh_\bR^*, \\
\Delta^\vee = &\;  \{ \alpha_i^\vee = e_i -e_{i+1} \mid i=1,\ldots, r-1\} \cup \{ \alpha_r^\vee = e_{r-1}+ e_r \} \subset \fh_\bR, \\
\Lambda = & \; \Big\{ \diag(\mu_1 J,\ldots, \mu_r J) =\sum_{j=1}^r \mu_j e_j  \; \big| \; \mu_j \in \bZ, \sum_{j=1}^r \mu_j \in 2\bZ \Big\},\\
\hLambda =& \; \pi_1 H = \bigoplus_{i=1}^r \bZ e_i.
\end{aligned}
$$
$$
\begin{aligned}
&  \varpi_{\alpha_i} =  \theta_1 +\cdots + \theta_i ,  \hspace{0.2cm}  i=1,\ldots,r-2; \\ 
& \varpi_{\alpha_{r-1}} = \frac{1}{2}(\theta_1+\cdots +\theta_{r-1} -\theta_r),\quad
\varpi_{\alpha_r} = \frac{1}{2} (\theta_1+\cdots +\theta_{r-1} +\theta_r).
\end{aligned} 
$$
In this case $G^{\ab}$ is trivial and $\pi_1 G= \pi_1 [G,G]=\langle e_r\rangle \simeq \bZ/2\bZ$.
 \end{example}

\begin{example} \label{Sp-roots} When $G=\Sp_r(\bC)=[G,G]$.
$$
\piH =\{ \diag(\mu_1,\ldots,\mu_r, -\mu_1,\ldots, -\mu_r): \mu_j \in \bZ\} \subset \fh_\bR =\{ \diag(\mu_1,\ldots,\mu_r, -\mu_1,\ldots,-\mu_r)\mid \mu_i \in \bR\}.
$$
We choose a $\bZ$-basis $e_1,\ldots, e_r$ of $\piH$ so that    $\diag(\mu_1,\ldots,\mu_r, -\mu_1,\ldots, -\mu_r) =\sum_{j=1}^r \mu_j e_j$.  
$$
\begin{aligned}
\Delta  =   &\;  \{ \alpha_i = \theta_i-\theta_{i+1} \mid i=1,\ldots,r-1\}  \cup \{ \alpha_r = 2\theta_r\} \subset \fh_\bR^*, \\
\Delta^\vee = &\;  \{ \alpha_i^\vee = e_i -e_{i+1} \mid i=1,\ldots, r-1\} \cup \{ \alpha_r^\vee = e_r\} \subset \fh_\bR, \\
\Lambda = & \; \hLambda =\pi_1 H =\bigoplus_{i=1}^r \bZ e_i.
\end{aligned}
$$
$$
\varpi_{\alpha_i} = \theta_1 +\cdots + \theta_i,  \quad i=1,\ldots, r. 
$$
In this case  $G^\ab$ is trivial and $\pi_1 G = \pi_1[G,G]  =\{1\}$.
\end{example}

Let $P_I \subset G$ be the standard parabolic subgroup  corresponding to $I\subset \Delta$. Let
$U_I \subset P_I$ (resp. $L_I\subset P_I$) be the unipotent radical (resp. Levi factor)  of $P_I$. Then
$L_I$ is a complex reductive Lie group. The coroot lattice of $L_I$ is the sublattice $\Lambda_I$ of $\Lambda$ generated
by $\{ \beta^\vee  \mid \beta \in \Delta -I \}$.
We have the following commutative diagram. 
\begin{equation}\label{diagram_with_fund_weights}
\begin{CD}
 0     @.  0\\
 @VVV @VVV\\
0 \ \longrightarrow\  \pi_1[L_I,L_I]=\hLambda_I  /\Lambda_I \  @>{j_{L_I}^{ss}}>> \pi_1[G,G]=\hLambda/\Lambda
@>{\oplus_{\alpha\in I}\varpi_\alpha}>> \oplus_{ \alpha\in I}\bQ/\bZ \\
  @V{i_{L_I} }VV @V{i_G}VV @|\\
  \pi_1 L _I=\pi_1 H/\Lambda_I @>{j_{L_I}}>> \pi_1 G =\pi_1 H /\Lambda
@>{\oplus_{\alpha\in I}\varpi_\alpha}>> \oplus_{\alpha\in I}\bQ/\bZ \\
 @V{p_{L_I}}VV @V{p_G}VV\\
 \pi_1 L_I^\ab =\pi_1(H)/\hLambda_I  @>{j_{L_I}^{\ab}}>> \pi_1 G^\ab=\pi_1 H/\hLambda \\
 @VVV @VVV\\
 0 @. 0 \\
\end{CD}
\end{equation}

\subsection{Topology of principal bundles over Riemann surfaces}  \label{topology} 
Let $G$ be a connected complex reductive Lie group.  Observe that the finite abelian group $\pi_1[G,G]$ is the torsion subgroup of the finitely generated abelian group $\piG$, and the quotient $\pi_1(G^\ab) = \piG/\pi_1[G,G]$ is a finitely generated {\em free} abelian group.  
\begin{definition} \label{baro2} 
Given a $C^\infty$ principal $G$-bundle
$\cP$ over a connected compact Riemann surface $X$, let $\bar{o}_2(\cP) \in \pi_1(G^\ab)$  be the image of
$o_2(\cP) \in \piG$ under the projection $\piG \to \pi_1(G^\ab) =\piG/\pi_1[G,G]$. 
\end{definition}

\smallskip

Let $\phi: G_1 \to G_2$ be a homomorphism between connected complex reductive Lie groups. We have the following commutative diagrams:
$$
\xymatrix{
G_1 \ar[r]^\phi  \ar[d]^{p_1} & G_2 \ar[d]^{p_2}  \\
G_1^\ab = G_1/[G_1,G_1] \ar[r]^{\phi^\ab} & G_2^\ab = G_2/[G_2,G_2] }\qquad
\xymatrix{ 
\piG_1 \ar[r]^{\phi_*} \ar[d]^{p_{1*}}  & \piG_2 \ar[d]^{p_{2*}} \\
\pi_1(G_1^\ab) \ar[r]^{\phi_*^\ab} & \pi_1(G_2^\ab) }
$$
where all the arrows are group homomorphisms, and the vertical arrows are surjective.  Let $\cP_1$ be a principal $G_1$-bundle over a connected compact Riemann surface $X$, so that
$$
o_2(\cP_1) \in \piG_1, \quad  \bar{o}_2(\cP_1)  = p_{1*} o_2(\cP_1) \in \pi_1(G_1^\ab).
$$
Let $\cP_2 = \cP_1\times_{G_1}  G_2= (\cP_1\times G_2)/G_1$, where $G_1$ acts on $\cP_2\times G_2$ on the right by 
$$
(p, g_2)\cdot g_1 = \left(p\cdot g_1, \phi(g_1^{-1}) g_2 \right) , \quad p\in \cP_1,  g_2\in G_2, g_2\in G_1.
$$
Then $\cP_2$ is a principal $G_2$-bundle over $X$, and 
$$
o_2(\cP_2) =\phi_* o_2(\cP_1) \in \piG_2, \quad \bar{o}_2(\cP_2) = p_{2*} o_2(\cP_2)= \phi_* \bar{o}_2(\cP_1) \in \pi_1(G_2^\ab). 
$$

Recall that $\piH/\widehat{\Lambda} \subset (\fz_G)_\bR$ is the dual lattice of the character lattice $\widehat{G} \subset (\fz_G)_\bR^*$. Therefore, 
the pairing $(\fz_G)_\bR^* \times (\fz_G)_\bR \to \bR$ restricts to a perfect pairing
$\langle\ ,  \rangle: \widehat{G} \times \piH/\hLambda \to \bZ$.  In particular $\mu =\bar{o}_2(P)$ can 
be viewed as an element in 
$\Hom(\widehat{G},\bZ)$, sending $\chi\in \widehat{G}$ to  $\langle \chi,\mu\rangle \in \bZ$. We  now provide an alternative
description of  $\langle \chi,\mu\rangle \in \bZ$ which does not rely on the inclusions $\pi_1 H/\hLambda \subset (\fz_G)_\bR$
and $\widehat{G} \subset (\fz_G)_\bR^*$.  Given  any 
$G$ character $\chi \in \widehat{G}=\Hom(G,\bC^*)$, let $L_\chi$ be the holomorphic line bundle over $X$ associated to the
1-dimensional representation $\chi: G\to \bC^*=\GL_1(\bC)$. Then
$$
\langle \chi,\mu\rangle = \deg(L_\chi) =\int_X c_1(L_\chi) \in \bZ. 
$$
Note that $\chi: G\to \bC^*$ is a special case of the group homomorphism $\phi: G_1\to G_2$ in Section \ref{topology}, and
$\chi_* : \piG \to \pi_1\bC^*\simeq \bZ$ sends $o_2(E)$ to $c_1(L_\chi)$. 

\subsection{The Harder-Narasimhan stratification}\label{HN_stratif_section}

Recall that, given a smooth projective curve $X$ and a connected complex reductive group $G$, we denote by $\BunGX$ the moduli stack of principal $G$-bundles on $X$ and, for all $d\in H^2(X;\piG)\simeq\piG$, we denote by $\BunGXd$ the moduli stack of principal $G$-bundles of degree $d$ on $X$. It is a connected smooth Artin stack of dimension $(g-1)\dim G$, where $g$ is the genus of $X$. 

\smallskip

The key notion for the present section is that of \textit{semistability} of a principal $G$-bundle on $X$, which is due to Ramanathan \cite{Ramanathan} and is recalled in Definition \ref{sst_def} below. The moduli stack of semistable principal $G$-bundles of degree $d$ on $X$ will be denoted by $\BunGXdss$. It is an open substack of $\BunGXd$. As we shall recall in Definition \ref{def_HN_type} and Proposition \ref{substack_fixed_HN_type}, the notion of semistability induces the existence of a decomposition of $\BunGXd$ into locally closed smooth algebraic substacks 
\begin{equation}\label{HN_stratif}
\BunGXd = \bigsqcup_{\nu\in\LGd} \Bmu
\end{equation} indexed by the possible \textit{Harder-Narasimhan types} of principal $G$-bundles of degree $d$ on $X$, and in which $\BunGXdss$ is the unique open substack.

\smallskip 

The topological type of a principal $G$-bundle $E$ over $X$ is classified
by its degree $d = o_2(E) \in \piG$ which determines  $\umu =\bar{o}_2(E) \in  \piG^{\ab} =\piH/\widehat{\Lambda}\subset (\fz_G)_\bR$ (cf. Definition 
\ref{baro2}).  As an example, when $G=\GL_r(\C)$ and $d\in \pi_1G\simeq \Z$, one has $\umu  =\frac{d}{r}I_r\in  (\fz_{\GL_r(\bC)})_\bR = \fz_{\GL_r(\bR)} =\bR I_r$ (cf. Example \ref{GLroots}). When $G$ is semisimple,  $G=[G,G]$ and $G^\ab$ is trivial, so $\mu=0$.

\smallskip

Let $E\to X$ be a principal $G$-bundle and let $P\subset G$ be a standard parabolic subgroup. Then
$E/P\to X$ is a fiber bundle with fiber $G/P$ and $E\to E/P$ is a principal $P$-bundle.  
The principal bundle $E\to X$ admits a reduction of the structure group from $G$ to the parabolic subgroup $P$ if and only the fiber bundle
$E/P\to X$ admits a section $\sigma: X\to E/P$. In this case, let $E_P\to X$ be the pullback
of $E/P\to E$: 
\begin{equation} \label{eqn:EP} 
\xymatrix{
E_P \ar[r] \ar[d] & E\ar[d] \\
X \ar[r]^\sigma & E/P
}
\end{equation} 
Then $E_P\to X$ is a principal $P$-bundle and $E = E_P\times_P G$.  

Given a representation $\bV$ of $P$, i.e. a group homomorphism $P\to \GL(\bV)$,  let $E(\bV) = E\times_P \bV \to E/P$ be the associated vector bundle. We have the following short exact sequence of vector bundles
over $E/P$:
\begin{equation}\label{eqn:EEE}
0 \to E(\fp) \lra  E(\fg) \lra E(\fg/\fp)\to 0.
\end{equation} 
The pullback of \eqref{eqn:EEE} under $\sigma:X\lra E/P$ is the following short exact sequence of vector bundles on $X$: 
\begin{equation}\label{eqn:EPE} 
0\to \sigma^* E(\fp) = \ad E_P \lra  \sigma^* E(\fg) = \ad E \lra \sigma^* E(\fg/\fp) = \sigma^* T_{G/P} \to 0
\end{equation} 
where $T_{G/P}$ is the tangent bundle along fibers of $E/P\to X$. It follows from \eqref{eqn:EPE} that 
\begin{equation}\label{eqn:degree} 
\deg (\ad E_P) + \deg (\sigma^* T_{G/P})  = \deg (\ad E) =0
\end{equation} 
where the second equality holds because  the structure group $G$  of $E$ is reductive.

\begin{definition}[Ramanathan \cite{Ramanathan}]\label{sst_def}
Let $G$ be a complex reductive group. A holomorphic principal $G$-bundle $E$ on $X$ is called \textit{stable} (resp. \textit{semistable}) if, for all maximal proper parabolic subgroup $P\subset G$ and all reduction of structure group $\sigma: X\lra E/P$, the following two equivalent conditions hold: 
\begin{enumerate}
\item[(i)] $\deg \left(\sigma^*T_{G/P} \right) > 0$ (resp. $\geq 0$),
\item[(ii)]  $\deg (\ad\, E_P) <0$ (resp. $\leq 0$), where $E_P$ is as in \eqref{eqn:EP}. 
\end{enumerate} 
\end{definition}
The conditions  (i) and (ii) are equivalent by Equation \eqref{eqn:degree}. (i) and (ii)  appear in Definition~1.1 
and Remark~2.2 in \cite{Ramanathan}, respectively. When $E$ is the frame bundle of a holomorphic vector bundle $V\to X$,  Definition \ref{sst_def} is equivalent to
the slope-stability (resp. slope-semistability) of $V$. 

\smallskip

Let $E$ be a principal $G$-bundle on $X$, and assume that $E$ admits a reduction $E_P$ to a standard parabolic subgroup $P\subset G$.
Let $U\subset P$ be its unipotent radical, which is a normal subgroup of $P$. The Levi factor  of $P$ is a reductive subgroup $L\subset P$ such that
$L \simeq P/U$.   
The topological type of $E_L:=E_P/U$ is characterized  by  $\delta= o_2(E_L)\in \piL$, which determines
an element $\mu = \bar{o}_2(L) \in \pi_1(L^{\ab}) \subset (\fz_L)_\bR := \fz_L \cap \fh_\bR \subset \fh_\bR$. 
We obtain a group homomorphism
\begin{equation}\label{eqn:muL}
\mu_L: \pi_1L\to \fh_\bR.
\end{equation}
with $\Ker \mu_L = \pi_1 [L,L]$.

\begin{example}\label{ex_of_HN_type}
The Levi factor of a standard parabolic subgroup of $\GL_r(\bC)$ is of the form
$$
L = \{\diag (A_1,\ldots, A_k) \mid A_i \in \GL_{r_i}(\bC),  r_1+\cdots + r_k =r\} \simeq \GL_{r_1}(\bC)\times \cdots \times  \GL_{r_k}(\bC). 
$$
$$
\ud = (d_1,\ldots, d_k) \in \pi_1 L^\ab \simeq \pi_1  (\bC^*)^k \simeq \bZ^k,\quad 
\umu=\diag \left( \frac{d_1}{r_1} I_{r_1},\cdots, \frac{d_k}{r_k} I_{r_k} \right). 
$$
\end{example}

\begin{remark}\label{sst_for_products}
When $L\simeq G_1\times G_2$ with each $G_i$ reductive, a principal $L$-bundle $E$ is isomorphic to the fiber product $E_1\times_X E_2$, where $E_1:=E/G_2$ is a principal $G_1$-bundle and $E_2:=E/G_1$ is a principal $G_2$-bundle on $X$. Since parabolic subgroups of $G_1\times G_2$ are of the form $P_1\times P_2$ for $P_i\subset G_i$ parabolic, it is readily checked that the principal $G$-bundle $E$ is semistable if and only if $E_1$ is semistable as a $G_1$-bundle and $E_2$ semistable as a $G_2$-bundle (a similar statement holds for stability). Moreover, as semistablility of a principal $\GL_r(\C)$-bundle $E$ is equivalent to the slope-semistability of the associated vector bundle $V:=E(\C^r)$, we can also think of a semistable principal $\GL_{r_1}(\C)\times\GL_{r_2}(\C)$-bundle of topological type $(d_1,d_2)$ as a direct sum $V_1\oplus V_2$ of slope-semistable vector bundles, where $V_i$ has rank $r_i$ and degree $d_i$ (a similar statement again holds for stability). Note that $V_1\oplus V_2$ is never stable as a vector bundle of rank $(r_1+ r_2)$, i.e.\ as a $\GL_{r_1+r_2}(\C)$-bundle, and that it is semistable as such a vector bundle if and only if $\frac{d_1}{r_1}=\frac{d_2}{r_2}$.
\end{remark}

If we denote by $\ov{C_0}$ the closure of the fundamental Weyl chamber 
$$
C_0:=\{X\in\fh_\bR \ |\ \forall\,\alpha\in\Delta,\alpha(X)>0\},
$$ 
we may assume that $\umu\in\ov{C_0}$. In Example \eqref{ex_of_HN_type}, this means that 
$$
\frac{d_1}{r_1}\geq \frac{d_2}{r_2} \geq \cdots \geq \frac{d_k}{r_k}. 
$$
 Now, given a principal $G$-bundle $E$ on $X$, Atiyah and Bott have proved the existence of a canonical reduction to a standard parabolic subgroup $P\subset G$ and satisfying the following conditions, which we phrase in the notation similar to that in \cite{FM} and \cite{Ho_Liu_YM_AMS}.

\begin{theorem}[Atiyah-Bott \cite{AB}]\label{HN_red}
Let $E$ be a principal $G$-bundle on $X$. There exists a unique standard parabolic subgroup $P\subset G$ and a unique reduction $E_P$ of $E$ to $P$ such that:
\begin{enumerate}
\item If $L$ denotes the Levi quotient of $P$, then the associated $L$-bundle $E_L:=E_P\times_P L$ is semistable.
\item If $I_P\subset \Delta$ is the set of simple roots corresponding to $P$ and $\umu\in(\fz_L)_\R \cap\ov{C_0}\subset\fh_\bR$ is the image of $\delta:=o_2(E_L)\in\piL$ under the canonical morphism $\piL\lra\pi_1(L^{\mathrm{ab}})\subset \fh_{\bR}$, then 
$$
\forall\,\alpha\in I_P,\ \alpha(\umu)>0\,.
$$
\end{enumerate}
\end{theorem}
For instance, if $\umu$ is of the form \eqref{ex_of_HN_type}, the second condition in Theorem \ref{HN_red} means that $\frac{d_1}{r_1}>\frac{d_2}{r_2}$.

\begin{definition}\label{def_HN_type}
Given a principal $G$-bundle $E$ on $X$, the unique pair $(P,E_P)$ satisfying the conditions of Theorem \ref{HN_red} is called the \textit{canonical reduction} of $E$, and the pair $(P,\umu)$ is called the \textit{Harder-Narasimhan type}, or \textit{instability type}, of $E$.
\end{definition}

Thus, the HN type (=Harder-Narasimhan type) of a $G$-bundle consists of a standard parabolic subgroup $P$ and an element $\umu\in (\fz_L)_{\R}\subset \fh_\bR$, where $L$ is the Levi factor of $P$. For instance, if $E$ is a semistable $G$-bundle of degree $d$, then $P=G$ and $\umu\in\pi_1G^{\mathrm{ab}}\subset  (\fz_G)_\bR$ is the image of $d\in\piG$ under the canonical morphism $\piG\lra\pi_1 G^{\mathrm{ab}}$. We shall denote this particular element of $\fh_\bR$ by $\umu_{ss}$. Note that, given $P$, an element $\umu\in (\fz_L)_{\R}$ must satisfy certain conditions for $\nu:=(P,\umu)$ to be the HN type of a principal $G$-bundle on $X$. These conditions make $\umu$ an \emph{Atiyah-Bott point} and are described in \cite[Lemma~2.1.2]{FM} and \cite[Lemma~3.1]{Ho_Liu_YM_AMS}: for a fixed $P$, the set of Atiyah-Bott points $\mu\in\ov{C_0}$ is discrete.

\begin{remark}\label{rel_btw_top_type_and_HN_type}
By \cite[Lemma~2.1.2]{FM}, given a principal $G$-bundle $E$ with canonical reduction $(P,E_P)$, the topological type $\ud$ of $E_L$ is entirely determined by the HN type $(P,\umu)$ and the degree $d$ of $E$. 
\end{remark}

\smallskip

For a fixed $d\in\piG$, let us denote by $\LGd$ the set of all possible HN types $\nu=(P,\umu)$ of principal $G$-bundles of degree $d$. There is a partial ordering on $\LGd$, defined by 
\begin{equation}\label{partial_ordering_on_HN_types}
(P,\umu) \leq (P',\umu')\ \mathrm{if}\ P\supset P'\ \mathrm{and}\ \umu'-\umu\in\Q_+R_+^\vee\ \mathrm{in}\ \fh_\bR,
\end{equation} where $\Q_+=\Q\cap[0;+\infty[$ and $R_+^\vee\subset\fh_\bR$ is the set of positive coroots. For instance, $(G,\umu_{ss})\leq (P,\mu)$ for all HN type $(P,\mu)$. In what follows, we will use the notation $\nu_{ss}:=(G,\umu_{ss})$.

\begin{remark}
The partial ordering \eqref{partial_ordering_on_HN_types} can also be described as follows:
\begin{equation}\label{Adams_ordering_on_HN_types}
(P,\umu) \leq (P',\umu')\ \mathrm{if}\ P\supset P'\ \mathrm{and}\ \umu\in \widehat{W\cdot\umu'}
\end{equation} where $W\simeq N(T)/T$ is the Weyl group of $G$ and $\widehat{W\cdot\umu'}$ is the convex hull of the Weyl orbit of $\umu'$ in 
$\Lie(T)=\sqrt{-1}\fh_{\bR}$. If one thinks of the closed Weyl chamber $\ov{C_0}\subset\fh_{\bR}$ as the set of $K$-conjugacy classes in $\Lie(K)=\sqrt{-1}\fg_{\bR}$, the condition $\umu\in \widehat{W\cdot\umu'}$ defines a partial ordering on this set, called the \textit{Adams ordering} in \cite{FM}. Thus, the Adams ordering on $\sqrt{-1}\fh_\bR/\Ad(K)$ induces the partial ordering \eqref{Adams_ordering_on_HN_types}, or equivalently \eqref{partial_ordering_on_HN_types}, on the set $\LGd$ of all HN types of $G$-bundles of topological type $d$, for all $d\in\piG$. 
\end{remark}

\begin{proposition}\label{substack_fixed_HN_type}
For all $\nu\in\LGd$, there exists an open substack $\Blnu \subset \BunGXd$ parameterizing principal $G$-bundles of HN type $\nu'\leq \nu$. As a consequence, there exists a locally closed substack $$\Bnu:=\Blnu \mathbin{\big\backslash} \bigcup_{\nu' < \nu} \fBun_{\leq\nu'}$$ of $\BunGXd$, parameterizing principal $G$-bundles of HN type $\nu$.
\end{proposition}

\begin{proof}
The proof is an adaptation of \cite[Lecture~5]{Heinloth_lectures}, from the vector bundle case to the principal bundle case. The key point is that, given a family $T\lra\BunGXd$ of principal $G$-bundles of degree $d$ parameterized by $T$, there exists a canonical decomposition $T=\cup_\nu T_\nu$ of $T$ into locally closed sub-schemes such that, for all geometric point $x$ of $T_\nu$, the corresponding $G$-bundle over $X$ is of HN type $\nu$, and for all geometric point $x'$ of $\ov{T_\nu}$, the corresponding HN type $\nu'$ satisfies $\nu'\geq\nu$.
\end{proof}

\smallskip

\begin{remark}
As we shall see in Formula \eqref{perfection_Poincare_polynomial}, the decomposition of $\BunGXd$ into $(\Blnu)_{\nu\in \LGd}$ is perfect in the sense that the associated Gysin long exact sequence breaks up into short sequences, which makes it possible to express the Poincaré series of $\BunGXd$ in terms of the Poincar\'{e} series of the substacks $\Blnu$. For $G$ reductive, the existence of a \textit{perfect stratification} in that sense was first obtained using the differential point of view of Remark \ref{differential_viewpoint}. Indeed, given a $C^\infty$ principal $G$-bundle $\cP$ on $X$, let us denote by $\cC_\nu$ the set of $(0,1)$-connections on $\cP$ such that the associated holomorphic principal $G$-bundle $E$ on $X$ is of HN type $\nu$, and call $\cC_\nu$ a HN \textit{stratum}. Atiyah and Bott defined a partial ordering $\cC_\nu\preccurlyeq \cC_{\nu'}$ on the set of HN strata, with the property that 
$ \cC_\nu\preccurlyeq \cC_{\nu'} $ if $\cC_{\nu'}\cap\ov{\cC_\nu}\neq \emptyset$, and they proved that if $\cC_\nu\preccurlyeq \cC_{\nu'}$, then $\nu\leq \nu'$ for the Adams ordering. In \cite{FM}, Friedman and Morgan proved that the converse is also true. Thus, the Adams ordering on the set of $K$-conjugacy classes in $\Lie(T) =\sqrt{-1}\fh_\bR$ induces a partial ordering on the set of HN strata of $\BunGXd$. This partial ordering satisfies the property that  $ \ov{\cC_{\nu}} \subset \cup_{\nu'{\geq \nu}} \cC_{\nu'} $. Moreover, each $\cC_\nu$ is of finite codimension $d_\nu$ in $\cC$ and, as a matter of fact, $(\nu\leq\nu') \Rightarrow (d_\nu\leq d_{\nu'})$. While the inclusion $\ov{\cC_{\nu}} \subset \cup_{\nu'{\geq \nu}} \cC_{\nu'}$ is not an equality in general (\cite[\S 4.3]{FM}), meaning that the closure of a stratum is not technically a union of higher strata, the decomposition 
$\cC=\sqcup_\nu \cC_\nu$ is still called a \textit{stratification} of $\cC$ in the literature. Each stratum is invariant under the action of the gauge group $\cG$ of $\cP$ and, in this context (using an equivariant Gysin sequence), Formula \eqref{perfection_Poincare_polynomial} takes the form 
$$
P_t\big(B\cG;\Q\big) = P_t^{\cG}\big(\cC;\Q\big) =  \sum_{\nu\in\LGd} t^{2d_\nu}\,P_t^{\cG}\big(\cC_{\nu};\Q\big)
$$ 
where $P_t^{\cG}(-)$ denotes the Poincaré series of the $\cG$-equivariant cohomology $\Q$-algebra $H^*_{\cG}(-;\Q)$. This is usually summed up by saying that the stratification $\cC=\sqcup_{\nu\in\LGd} \cC_\nu$ is $\cG$-equivariantly perfect.
\end{remark}

\subsection{Codimension of the strata}

We now recall how to compute the codimension of $\Bnu$ in $\Blnu$, or equivalently $\BunGXd$. To do it, it suffices to compute the rank of the normal bundle to $\Bnu$. Let $E$ be a principal $G$-bundle of HN type $\nu=(P,\umu)$ and denote by $E_P$ the canonical reduction of $E$ to $P$ of Theorem \ref{HN_red}. As recalled in \eqref{infinitesimal_deformations}, the fiber at $E$ of the tangent stack to $\BunGXd$ is the quotient stack 
\begin{equation}\label{fiber_of_tgt_bundle_to_mod_stack}
[H^1(X;\ad\,E)/H^0(X;\ad\,E)]\,,
\end{equation} where $H^0(X;\ad\,E)$ acts trivially on $H^1(X;\ad\,E)$. To determine the fiber at $E$ of the normal bundle to $\Bnu$ in $\Blnu$, consider the short exact sequence of $P$-modules $$0\lra \fp \lra \fg \lra \fg/\fp \lra 0$$ and the induced short exact sequence of vector bundles 
$$
0\lra E_P\times_P\fp \lra E_P\times_P\fg \lra E_P\times_P(\fg/\fp) \lra 0\,
$$ which can also be written
\begin{equation}\label{ad_E_P_perp}
0\lra \ad(E_P) \lra \ad(E) \lra \ad(E)/\ad(E_P)\lra 0\,.
\end{equation}
The associated cohomology exact sequence then reads:
$$
\ldots\,\lra H^0\big(X; \ad(E)/\ad(E_P)\big) \lra H^1\big(X;\ad(E_P)\big) \lra H^1\big(X;\ad(E)\big) \lra H^1\big(X;\ad(E)/\ad(E_P)\big) \lra 0\,.
$$ We deduce from this and \eqref{fiber_of_tgt_bundle_to_mod_stack} that the fiber at $E$ of the normal bundle to $\Bnu$ in $\BunGXd$ is the quotient stack 
\begin{equation}\label{fiber_or_normal_bundle}
\left[H^1\big(X;\ad(E)/\ad(E_P)\big)\, \big/\, H^0\big(X;\ad(E)/\ad(E_P)\big)\right]\,
\end{equation} where $H^0(X;\ad(E)/\ad(E_P))$ acts trivially on $H^1(X;\ad(E))$.  As we shall see in Lemma \ref{towards_formula_for_codim}, one actually has $H^0(X;\ad(E)/\ad(E_P))=0$, so the fiber at $E$ of the normal bundle to $\Bnu$ is in fact just the vector space $H^1(X;\ad(E)/\ad(E_P))$. For now, we deduce from \eqref{fiber_or_normal_bundle} that the rank of the normal bundle to $\Bnu$ is equal, by Riemann-Roch, to
\begin{eqnarray*}
& & \dim H^1\big(X;\ad(E)/\ad(E_P)\big) - \dim H^0\big(X;\ad(E)/\ad(E_P)\big) \\
& = & -\deg\big(\ad(E)/\ad(E_P)\big) + (g-1)\mathrm{rk}\big(\ad(E)/\ad(E_P)\big)\\
& = & \deg\big(\ad(E_P)\big) + (g-1)\dim(\fg/\fp)
\end{eqnarray*} where the last inequality follows from the short exact sequence \eqref{ad_E_P_perp}, the fact that $\deg\,\ad(E)=0$, and the fact that $\ad(E)/\ad(E_P) \simeq E_P\times_P(\fg/\fp)$.

\begin{lemma}\label{towards_formula_for_codim}
Let $E$ be a principal $G$-bundle of HN type $\nu=(P,\mu)$ on $X$, and let $E_P$ be its canonical reduction. Then:
\begin{enumerate}
\item $\dim(\fg/\fp) = \#\,\big\{\alpha\in R_+\ |\ \alpha(\umu)>0\big\}$.
\item $\deg\,\ad(E_P) = \sum_{\alpha\in R_+\ |\ \alpha(\umu)>0}\alpha(\umu)$.
\item $H^0(X;\ad(E)/\ad(E_P))=0$.
\end{enumerate}
\end{lemma}

\begin{proof}
Let us prove these results.
\begin{enumerate}
\item Recall the root space decomposition $\fg=\fh\oplus \bigoplus_{\alpha\in R}\fg_\alpha\,.$ Let $I\subset\Delta$ be the subset of the set of simple roots that corresponds to $\fp$ in the sense that $\fp=\fh\oplus\bigoplus_{\alpha\in\Gamma}\fg_\alpha\,,$ where $\Gamma = R_+ \cup\{\alpha\in R\ |\ \alpha\in\mathrm{span}_\Z(\Delta-I)\}$. Equivalently, $\fg/\fp \simeq \bigoplus_{\alpha\notin\Gamma}\fg_\alpha$. Since the root spaces are $1$-dimensional over $\C$, one has $$\dim\fp\ =\ \dim\fh\, +\, \#\,\{\alpha\in R\ |\quad \alpha\in\Gamma\}\quad \text{and}\quad \dim\fg-\dim\fp\ =\ \#\,\{\alpha\in R\ |\ \alpha\notin\Gamma\}\,.$$ But, by definition of the HN type (Theorem \ref{HN_red}), if $E$ of HN type $(P,\mu)$, we have $\mu\in\ov{C_0}$, so for all $\alpha\in\Delta$, $\alpha(\mu)\geq 0$, and, more precisely: $$\forall\, \alpha\in I, \alpha(\mu)>0\quad \mathrm{and}\quad\forall\,\alpha\in\Delta-I, \alpha(\mu)=0,$$ where the second condition follows from the fact that $\mu\in\fz_{L_\R}$ (by definition of $\mu$). Equivalently, $\alpha\in\Gamma$ if and only if $\alpha\in R_+$ or $\alpha(\mu)=0$. So, $\alpha\notin\Gamma$ if and only if $\alpha\in R_-$ and $\alpha(\mu)<0$. Therefore:
$$
\dim\fg-\dim\fp\ =\  \#\,\{\alpha\in R_-\ |\ \alpha(\mu)<0\} \ =\ \#\,\{\alpha\in R_+\ |\ \alpha(\mu)>0\}\,.
$$
\item 
We have 
$$
\fp = \fh \oplus \bigoplus_{\alpha \in \Gamma} \fg_\alpha, 
$$
so $H$ acts on $\det \fp = \Lambda^{\dim \fp} \fp \simeq\bC$ by the character
$$
\chi:= \sum_{\alpha \in \Gamma} \alpha   \in \widehat{H} =\Hom(H,\bC^*). 
$$
The $H$-action on $\fp$ is the restriction of the adjoint representation of $P$, so $\det \fp$ is a one-dimensional representation of 
$P$, and $\chi$ is contained in the sublattice $\widehat{P}=\widehat{L}\subset \widehat{H}$. 
$$
\begin{aligned} 
\deg\, \ad(E_P) = &  \deg  \det (\ad(E_P)) = \deg (E_P \times_P \det\fp) = \deg(E_L \times_L \det \fp) \\
= &  \langle \chi, \mu\rangle =\sum_{\alpha\in \Gamma} \alpha(\mu) =\sum_{\alpha \in R_+ \ |\ \alpha(\mu)>0} \alpha(\mu).
 \end{aligned} 
 $$
\item Recall that if $V$ is a vector bundle with Harder-Narasimhan filtration $$0=V_0\subset V_1\subset\,\ldots\, \subset V_\ell =V\,,$$ (such that the $V_i/V_{i-1}$ are semistable and the slopes $\mu_i:=\mu(V_i/V_{i-1})$ satisfy $\mu_1>\,\ldots\,>\mu_\ell$), one denotes $\sup\,V :=\mu_1$ and $\inf\,V:=\mu_\ell$. Moreover, if $u:W \lra V$ is a morphism of vector bundles on $X$, then $(\sup\,V<\inf\,W) \Rightarrow (u=0)$. Indeed, a morphism between two semistable bundles is zero if the slope of the target is strictly smaller than the slope of the source, which implies the desired result by double induction on the slopes of the HN filtrations (\cite[Lemma~10.2]{AB}). So, to prove that $H^0(X;V)=0$, it suffices to show that $\sup \,V<0$, since this will imply that $\Hom(O_X,V)=0$. Let us now apply this to $V=\ad(E)/\ad(E_P)$. As $G$ is reductive, the adjoint bundle $\ad(E)$ carries a non-degenerate quadratic form $\kappa$, so it is a self-dual bundle, and its Harder-Narasimhan filtration is of the form 
$$
0\subsetneq F_{-k} \subset F_{1-k} \subset\,\ldots\,\subset F_{-1}\subset F_0 \subset F_1 \subset \,\ldots\, \subset F_k =\ad(E)
$$ 
with $F_{-j} = F_{j-1}^{\perp_\kappa}$ and $F_{1-j}/F_{-j} \simeq (F_j/F_{j-1})^\vee$. By construction of the canonical reduction in \cite{AB}, $F_0=\ad(E_P)$. In particular, the HN filtration of $\ad(E)/\ad(E_P)$ is induced by that of $\ad(E)$, from which we get that 
$\sup \ad(E)/\ad(E_P) = \mu(F_1/F_0)  = -\mu(F_0/F_{-1}) <0$.  We conclude that $H^0(X, \ad E/\ad E_P)=0$. 
\end{enumerate}
\end{proof}

From Lemma \ref{towards_formula_for_codim} and the discussion preceding it, one deduces the following result, which is due to Atiyah and Bott.

\begin{proposition}\cite[p.~592]{AB}
Let $\nu=(P,\umu)\in\LGd$ be a HN type for $G$-bundles of degree $d$. For all $E\in\Bnu$, the dimension of $H^1(X;E_P\times_P(\fg/\fp))$ depends only on $\nu$ and is given by the formula 
\begin{equation}\label{codim_of_stratum}
d_\nu:=\sum_{\alpha\in R_+\ |\ \alpha(\umu)>0} \big(\alpha(\umu)+(g-1)\big).
\end{equation} Formula \eqref{codim_of_stratum} gives the codimension of $\Bnu$ in $\Blnu$, or equivalently $\BunGXd$.
\end{proposition}

As one sees from Formula \eqref{codim_of_stratum}, the unique open stratum of $\BunGXd$ is $\BunGXdss$.

\subsection{The Gysin exact sequence}

We refer to \cite[Corollary~2.1.3]{Behrend_trace_formula_Inventiones} for a general approach to Gysin sequences for smooth algebraic stacks. The special case that we need is the following, in which we denote by $\Bslnu$ the open substack $$\Bslnu:= \bigsqcup_{\nu'<\nu}\fBun_{\leq\nu'}.$$ The results of \cite{Behrend_trace_formula_Inventiones} indeed apply because the canonical morphisms $\Bnu\hookrightarrow \Blnu$ are closed immersions of smooth stacks.

\begin{proposition}
For all HN type $\nu\in\LGd$, consider the open substacks $\Bslnu\subset\Blnu\subset\BunGXd$, and denote by $d_\nu$ the complex codimension of $\Bnu=\Blnu\setminus\Bslnu$ in $\BunGXd$. Then the canonical morphism $\Bslnu\overset{i_\nu}{\hookrightarrow}\Blnu$ induces a long exact sequence of $\C$-vector spaces
\begin{equation}\label{Gysin_sequence}
\ldots\,\lra H^{k-2d_\nu}\big(\Bnu;\C\big) \lra H^{k}\big(\Blnu;\C\big) \overset{i_\nu^*}{\lra} H^k\big(\Bslnu;\C\big)\lra\,\ldots\ .
\end{equation}
\end{proposition}

The Gysin sequence \eqref{Gysin_sequence} holds with rational coefficients, too, but we write it with complex coefficients to emphasize that all $\C$-vector spaces in this sequence have a pure Hodge structure. As we shall see in \eqref{comp_Gysin_sequence_Hodge_structures}, the $\C$-linear maps in \eqref{Gysin_sequence} are compatible with these Hodge structures. 

\smallskip

For now, recall that the Atiyah-Bott strategy to prove that the long exact sequence \eqref{Gysin_sequence} breaks up into short exact sequences of $\C$-vector spaces
\begin{equation}\label{perfection}
0\lra H^{k-2d_\nu}\big(\Bnu;\C\big) \lra H^{k}\big(\Blnu;\C\big) \overset{i_\nu^*}{\lra} H^k\big(\Bslnu;\C\big)\lra 0
\end{equation} is to prove that the Euler class $e(\cN_\nu)$ of the normal bundle to $\Bnu$ is not a zero divisor in the $\C$-algebra $H^*(\Bnu;\C)$, so that taking the cup product with this class yields an injective map. Moreover, there is a commutative diagram 
\begin{equation}\label{Euler_class}
\begin{tikzcd}
H^{k-2d_\nu}\big(\Bnu;\C\big) \arrow[r] \arrow[rd, "\cdot\,\cup e(\cN_\nu)"'] & H^k\big(\Blnu;\C\big) \ar[d, "\mathrm{restr.}"] \\
& H^k\big(\Bnu;\C\big)
\end{tikzcd}
\end{equation} which proves that the Gysin exact sequence \eqref{Gysin_sequence} breaks up into short exact sequences \eqref{perfection}. Now, all $\C$-vector spaces in Diagram \eqref{Euler_class} carry a pure Hodge structure, and the complex Euler class of $\cN_\nu$ is of type $(d_\nu,d_\nu)$, so the cup product with $e(\cN_\nu)$ is compatible with the Hodge structures, in the following sense (\cite[Proposition~8.2.11]{Deligne_Hodge_III}).

\begin{lemma}\label{AB_lemma_Hodge_case}
There are short exact sequences
\begin{equation}\label{Hodge_perfection}
0\lra H^{p-d_\nu,q-d_\nu}\big(\Bnu\big) \lra H^{p,q}\big(\Blnu\big) \overset{i_\nu^*}{\lra} H^{p,q}\big(\Bslnu\big)\lra 0\ .
\end{equation}
\end{lemma}

\begin{proof}
Since the Euler class $e(\cN_\nu)$ is of type $(d_\nu,d_\nu)$, we have a commutative diagram
\begin{equation}\label{mult_by_Euler_class}
\begin{tikzcd}
H^{p-d_\nu,q-d_\nu}\big(\Bnu\big) \arrow[r] \arrow[rd, "\cdot\,\cup e(\cN_\nu)"'] & H^{p,q}\big(\Blnu\big) \ar[d, "\mathrm{restr.}"] \\
& H^{p,q}\big(\Bnu\big)
\end{tikzcd}
\end{equation}
and since this Euler class is not a zero divisor in the $\C$-algebra $H^*(\Bnu;\C)$, the first morphism in the following piece of the Hodge-Gysin exact sequence
\begin{equation}\label{Hodge_Gysin}
\ldots\,\lra H^{p-d_\nu,q-d_\nu}\big(\Bnu\big) \lra H^{p,q}\big(\Blnu\big) \overset{i_\nu^*}{\lra} H^{p,q}\big(\Bslnu\big)\lra\,\ldots
\end{equation} is an injective map, which proves the result.
\end{proof}

\begin{remark}
The Hodge-Gysin exact sequence \eqref{Hodge_Gysin} is obtained as follows. For convenience in the notation, we set $X:=\Blnu$ and $Y:=\Bnu$. Note that the substack $Y$ is closed in $X$. So, by \cite[(7)~p.138]{Behrend_trace_formula_Inventiones}, if we consider a sheaf of Abelian groups $\mathscr{F}$ on $X$, there is a long exact sequence 
\begin{equation}\label{cohomology_with_supports}
\ldots\,\lra H^k_Y(X;\mathscr{F}) \lra H^k(X;\mathscr{F}) \overset{i_U^*}{\lra} H^k(U;\mathscr{F}|_U) \lra \ldots
\end{equation} where $U:=X\setminus Y$ is open in $X$, with canonical inclusion $i_U:U\hookrightarrow X$, and $H^k_Y(X;\mathscr{F})$ is the $k$-th right derived functor of $\Ga_Y(X;-)$, the functor of global sections with support in $Y$. Note that when $\mathscr{F}=\C$, we can also think of cohomology groups with coefficients in $\mathscr{F}$ as de Rham cohomology groups. Recall now that the Thom class of the normal bundle $N_{Y/X}$ induces an isomorphism $H^{k-2d}(Y;\C)\overset{\simeq}{\lra} H^k_{Y}(X;\C)$, where $d=\mathrm{rank}_{\C}(N_{Y/X}) = \mathrm{codim}_{\C}(Y/X)$. As the Thom class is of type $(d,d)$, the Thom isomorphism induces isomorphisms $$H^{p-d,q-d}(Y)\overset{\simeq}{\lra} H^{p,q}_Y(X)$$ for all $(p,q)$ such that $p+q=k$. So the long exact sequence \eqref{cohomology_with_supports} induces long exact sequences 
\begin{equation}\label{comp_Gysin_sequence_Hodge_structures}
\ldots\,\lra H^{p-d,q-d}(Y) \lra H^{p,q}(X) \overset{i_U^*}{\lra} H^{p,q}(U) \lra \ldots
\end{equation} and, when $X=\Blnu$, $Y=\Bnu$ and $d=d_\nu$, we get indeed \eqref{Hodge_Gysin} from \eqref{comp_Gysin_sequence_Hodge_structures}. In sum, the morphisms in the Gysin sequence \eqref{Gysin_sequence} are compatible with the Hodge structures on each term in the sequence, and this is how one obtains the Hodge-Gysin exact sequence \eqref{Hodge_Gysin}. Also, the commutativity of Diagram \eqref{mult_by_Euler_class} follows from the fact that the Euler class is the pullback of the Thom class by the zero section of the normal bundle.
\end{remark}

Lemma \ref{AB_lemma_Hodge_case} implies that the perfection of the Harder-Narasimhan stratification is compatible with the Hodge structure of $H^*(\BunGXd;\C)$. In the Atiyah-Bott case, the fact that the Gysin sequence \eqref{Gysin_sequence} breaks up into short exact sequences \eqref{perfection} implies that the Betti numbers of $\BunGXd$ can be computed from the Betti numbers of its HN strata: 
\begin{equation}\label{perfection_Poincare_polynomial}
P_t\big(\BunGXd;\C\big) = \sum_{\nu\in\LGd} t^{2d_\nu}\,P_t\big(\Bnu;\C\big)\,.
\end{equation} This implies that only those strata whose codimension is lower than $\frac{j}{2}$ contribute to the cohomology in degree $j$: $$\forall\, j\geq 0,\ H^j\big(\BunGXd;\C\big) = H^j\big(\cup_{\nu\,|\,d_\nu\leq \frac{j}{2}} \Bnu;\C\big)\,.$$ Indeed, if $d_\nu> \frac{j}{2}$, the short exact sequence \eqref{perfection} reduces to an isomorphism $ H^{j}(\Blnu;\C) \simeq H^j(\Bslnu;\C)$. Note that the set $\{\nu\in \LGd\,|\, d_\nu\leq\frac{j}{2}\}$ is finite (\cite[(1.18)]{AB}). 

\smallskip

In view of Lemma \ref{AB_lemma_Hodge_case}, Formula \eqref{perfection_Poincare_polynomial} can be refined in the following way.
 
\begin{theorem}\label{perfection_thm}
For all $d\in\piG$, let $\LGd$ be the set of Harder-Narasimhan types of principal $G$-bundles on $X$. Then the Hodge-Poincaré polynomial of $\BunGXd$ satisfies $$HP_{u,v}\big(\BunGXd\big)\ \ = \sum_{\nu\in\LGd} (uv)^{d_\nu}\,HP_{u,v}\big(\Bnu\big)$$ where $\Bnu\subset\BunGXd$ is the stack of $G$-bundles of HN type $\nu$ on $X$ and $d_\nu$ is given by Formula \eqref{codim_of_stratum}.
\end{theorem}

\subsection{Hodge-Poincaré series of moduli stacks of semistable bundles}

In this subsection, we obtain a recursive formula for the Hodge-Poincaré series of the moduli stack $\BunGXdss$.

\smallskip

First, we see from Theorem \ref{perfection_thm}, that 
\begin{equation}\label{formula_before_inductive_step}
HP_{u,v}\big(\BunGXdss\big)\ \ =\ \ HP_{u,v}\big(\BunGXd\big)\quad - \sum_{\nu\in\LGd\setminus\{\nu_{ss}\}} (uv)^{d_\nu}\,HP_{u,v}\big(\Bnu\big).
\end{equation} The inductive step consists in expressing $HP_{u,v}(\Bnu)$, for $\nu\neq\nu_{ss}$, in terms of the Hodge-Poincaré series of a moduli stack of \emph{semistable} bundles for a certain reductive subgroup $L\subset G$. Let $\nu=(P,\umu)\in\LGd$ be a HN type for $G$-bundles of degree $d$. As recalled in Remark \ref{rel_btw_top_type_and_HN_type}, if $E$ is a principal $G$-bundle of degree $d$ and HN type $(P,\mu)$, the principal $L$-bundle $E_L:=E_P\times _P L$, where $L$ is the Levi quotient of $P$, has a well-defined topological type $\ud_\mu\in\pi_1L$ (since $d$ is fixed throughout, we do not indicate the dependence of $\delta_\mu$ on $d$ in the notation). So, by Theorem \ref{HN_red}, there is a morphism of algebraic stacks
\begin{equation}\label{inductive_step}
\begin{array}{ccc}
\Bnu = \fBun_{(P,\,\umu)} & \lra & \BunLXmuss \\
E & \lmt & E_L:=E_P\times_P L
\end{array}
\end{equation} where $\BunLXmuss$ denotes the moduli stack of semistable $L$-bundles of topological type $\ud_\mu$ on $X$ (since $L=P/R_u(P)$, where $R_u(P)$ is the unipotent radical of $P$, one also has $E_L = E_P/R_u(P)$). It is now a consequence of the results of Atiyah and Bott that \eqref{inductive_step} is an acyclic fibration. Since Hodge structures are functorial and compatible with products (\cite{Deligne_Hodge_III}), we have $HP_{u,v}(\Bnu) = HP_{u,v}(\BunLXmuss)$. Thus, Formula \eqref{formula_before_inductive_step} becomes the following recursive formula:
\begin{equation}\label{recursive_formula_for_HP_series}
HP_{u,v}\big(\BunGXdss\big)\ \ =\ \ HP_{u,v}\big(\BunGXd\big)\quad - \sum_{\nu=(P,\,\umu)\in\LGd\setminus\{\nu_{ss}\}} (uv)^{d_\nu}\,HP_{u,v}\big(\BunLXmuss\big)
\end{equation} where:
\begin{itemize}
\item $HP_{u,v}(\BunGXd)$ is given by Theorem \ref{thm:Hodge},
\item $d_\nu$ is given by Formula \eqref{codim_of_stratum} and, 
\item for all HN type $\nu=(P,\umu)$, we let $L:=P/R_u(P)$ be the Levi quotient of the parabolic subgroup $P$, which can be seen as a reductive subgroup of $G$, and then $HP_{u,v}\big(\BunLXmuss\big)$ is given by Formula \eqref{HP_series_of_BunLXd}.
\end{itemize} Formula \eqref{recursive_formula_for_HP_series} is a refinement of the usual formula for the Poincaré series of $\BunGXdss$ in the sense that, if we set $u=v=t$ in \eqref{recursive_formula_for_HP_series}, we get indeed, as in \cite{AB}:
$$
P_{t}\big(\BunGXdss\big)\ \ =\ \ P_{t}\big(\BunGXd\big)\quad - \sum_{\nu=(P,\,\umu)\in\LGd\setminus\{\nu_{ss}\}} t^{2d_\nu}\,P_{t}\big(\BunLXmuss\big)\,.
$$

In order to apply Formula \eqref{recursive_formula_for_HP_series} to compute $HP_{u,v}\big(\BunGXdss\big)$ recursively, it suffices to have a concrete understanding of $\LGd$ for all all reductive group $G$. For classical groups, one can for instance consult \cite[\S 3.4]{Ho_Liu_YM_AMS}.

\begin{example}\label{example_rank_2_vb}
When $G=\GL_2(\C)$, we have $\piG\simeq \Z$ and, tensoring by a line bundle if necessary, we only need to consider the cases when $d=0$ and $d=1$. With respect to the standard upper triangular Borel subgroup $B=\mathbf{Tsup}_2(\C)$, the possible HN types for a $\GL_2(\C)$-bundle of degree $d$ are 
$$\nu_{ss} = \big(\GL_2(\C),d\big)\ \ \mathrm{and}\ \ \nu_k=\big(\mathbf{Tsup}_2(\C),(k,d-k)\big),\ k>\frac{d}{2}\,.$$ If $\nu\neq \nu_{ss}$, the associated Levi subgroup is the maximal torus $L=T\simeq\C^*\times\C^*$, so all $L$-bundles are semistable in this case. When $d=0$, one has $d_{\nu_k} = 2k+g-1$ for all $k\geq 1$, so Formula \ref{recursive_formula_for_HP_series} gives 
$$HP_{u,v}\big(\fBun_{\GL(2;\C)}^{\, 0,ss}\big) = \frac{(1+u)^g(1+v)^g}{1-uv}\, \frac{(1+u^2v)^g(1+uv^2)^g}{(1-uv)(1-u^2v^2)} - \frac{(uv)^{g+1}}{1-u^2v^2}\,\left(\frac{(1+u)^g\,(1+v)^g}{1-uv}\right)^2$$ while, for $d=1$, one has $d_{\nu_k} = 2k+g-2$ for all $k\geq 1$, so Formula \ref{recursive_formula_for_HP_series} gives 
\begin{equation}\label{towards_HP_polyn_of_mod_space}
HP_{u,v}\big(\fBun_{\GL(2;\C)}^{\, 1,ss}\big) = \frac{(1+u)^g(1+v)^g}{1-uv}\, \frac{(1+u^2v)^g(1+uv^2)^g}{(1-uv)(1-u^2v^2)} - \frac{(uv)^g}{1-u^2v^2}\,\left(\frac{(1+u)^g\,(1+v)^g}{1-uv}\right)^2\,.
\end{equation} As we shall see in Example \ref{example_rank_2_vb_bis}, the Formula for $d=1$ (for the moduli \emph{stack} of semistable vector bundles of rank $2$ and degree $1$) is consistent with the results of Earl and Kirwan in \cite{Earl_Kirwan} (for the moduli \emph{space} of semistable vector bundles of rank $2$ and degree $1$).
\end{example}

\section{Inversion of recursive formulae}\label{inversion_section}

\subsection{The closed formula}

Recall that, given a smooth projective curve $X$, a connected complex reductive group $G$ and an element $d\in H^2(X;\piG)\simeq\piG$, we denote by $\BunGXdss$ the moduli stack of semistable principal $G$-bundles of degree $d$ on $X$. It is a connected smooth Artin stack of dimension $(g-1)\dim G$, where $g$ is the genus of $X$. We fix a maximal torus $H\subset G$ and denote by $R$ the associated root system. We also choose a Borel subgroup $B\supset H$ and denote by $R_+\subset R$ the associated set of positive roots. This determines a set of simple roots $\Delta\subset R_+$. For all $\alpha\in R$, the corresponding coroot is denoted by $\alpha^\vee$. And for all $\alpha\in \Delta$, the associated fundamental weight is denoted by $\varpi_\alpha$. We also fix a maximal compact subgroup $K\subset G$ and we let $T$ be the compact maximal torus $H\cap K$. As in Section \ref{HN_stratif_section}, we view $\pi_1H$ as a lattice in $\sqrt{-1}\ft$, where $\ft=Lie(T)$, and we denote by $\Lambda\subset \pi_1H$ the sub-lattice generated by the simple coroots $\Delta^\vee$.
We let $\widehat{\Lambda} := (\Lambda \otimes_{\bZ} \bQ)\cap \pi_1 H$ be the saturation of $\Lambda$ in $\pi_1 H$.

\smallskip

Let us now consider the fundamental weights $\{\varpi_\alpha \mid \alpha\in \Delta\}$ of $G$. By definition, the family $\{\varpi_\alpha \mid \alpha\in \Delta\}$ is a basis of
the real vector space $\Hom(\Lambda\otimes_\Z\R;\R)$ which is dual to the basis
$\Delta^\vee=\{ \alpha^\vee \mid \alpha\in\Delta\}$ of the real vector space $\Lambda\otimes_\Z\R\subset \sqrt{-1}\ft$. Given $\alpha\in \Delta$, we extend $\varpi_\alpha:\Lambda\otimes_\Z\R \lra \bR$ to $\varpi_\alpha:\i\ft= (\Lambda\otimes_\Z\R) \oplus \i\fz_{G_\R} \lra \bR$ by zero on $\i\fz_{G_\R}$, where $\fz_{G_\R}=\mathrm{Z}(\fg)\cap\ft$. Then $\varpi_\alpha$ takes integral values on the lattice $\Lambda=\oplus_{\alpha\in \Delta} \bZ \alpha^\vee\subset \Lambda\otimes_\Z\R \subset \sqrt{-1}\ft$, and rational values on the larger lattice $\pi_1H \subset \sqrt{-1}\ft$.

\smallskip

In Section \ref{proof_of_inversion}, we will prove the following closed formula for the Hodge-Poincar\'{e} series of moduli stacks of semistable principal $G$-bundles on $X$.

\begin{theorem} \label{thm:HP}
For all $d\in\piG$, the Hodge-Poincaré series of the moduli stack $\BunGXdss$ is equal to the expansion in $\Z[\![u,v]\!]$ of the element of $\Q(u,v)$ defined by the right-hand side of the following formula:
$$
HP_{u,v}(\BunGXdss)  = 
\sum_{I\subseteq \Delta}
(-1)^{\dim Z(L^I)-\dim Z_G}
a_{u,v}(L^I_X)
\frac{(uv)^{(g-1)\dim U^I}}
{\prod_{\alpha\in I}
\big(1-(uv)^{2\varrho^I(\alpha^\vee)}\big)}\cdot
(uv)^{2\sum_{\alpha\in I}\varrho^I(\alpha^\vee)
\langle \varpi_\alpha(d)\rangle}
$$
where, for all subset $I\subset \Delta$ of the set of simple roots of $G$, we denote by:
\begin{itemize}
\item $P^I$ the parabolic subgroup of $G$ associated to $I\subset\Delta$, 
\item $L^I$ the Levi quotient of $P^I$,
\item $U^I$ the unipotent radical of $P^I$, 
\item $a_{u,v}(L^I_X)$ the Hodge-Poincaré series 
$$
a_{uv}(L^I_X) =  \left(\frac{ (1+u)^g(1+v)^g }{1-uv} \right)^{\dim Z(L^I)} \prod_{k=\dim Z(L^I)+1}^{\mathrm{rk}(L^I)}
 \frac{(1+u^{d_k(L^I)} v^{d_k(L^I)-1})^g (1+u^{d_k(L^I)-1} v^{d_k(L^I)})^g}{\big(1- (uv)^{d_k(L^I)-1}\big)\big(1-(uv)^{d_k(L^I)}\big)}\,.
$$
of the moduli stack of all principal $L^I$-bundles on $X$ of a fixed degree $\delta\in \pi_1 L_I$, as computed in \eqref{HP_series_of_BunLXd},
\item $\varrho^I\in\fh^*$ the linear form on the Cartan sub-algebra $\fh\subset \fg$ defined by 
$$
\varrho^I\quad =\quad \frac{1}{2}\sum_{\tiny \begin{array}{c}\beta\in R_+\ |\ \exists\,\alpha\in I\\
 \langle \beta,\alpha^\vee \rangle >0\end{array}} \beta
$$
\end{itemize}
and, moreover,
\begin{itemize}
\item $\varpi_\alpha(d)\in \bQ/\bZ$ is the image of $d$, as defined in Diagram \eqref{diagram_with_fund_weights}, ,
\item and $\langle x\rangle \in\bQ$ is the unique representative of the class $x\in \bQ/\bZ$ such that $0< \langle x\rangle\leq 1$.
\end{itemize}
\end{theorem}

\subsection{Proof of Theorem \ref{thm:HP}}\label{proof_of_inversion}

\subsubsection{Notation}
The following is  a correspondence between the notation in \cite{FM} and that in \cite{LR}.
\\

\begin{center}
\begin{longtable}{||c|c|c||}\hline
                           & \cite{LR}    & \cite{FM} \\ \hline
\begin{tabular}{c}
minimal parabolic \\
subgroup (Borel)
\end{tabular}               & $P_0$        &         \\ \hline
Cartan of $G$               & $M_0$        & $H$       \\ \hline
parabolic subgroup          & $P=M_P N_P$  & $P=LU$    \\ \hline
Levi subgroup               & $M_P$        & $L$       \\ \hline
unipotent radical           & $N_P$        & $U$       \\ \hline
\begin{tabular}{c}
center of the \\
Levi subgroup
\end{tabular}               & $Z_P$        & $Z(L)$    \\ \hline
\begin{tabular}{c}
connected center \\
of $M_P$
\end{tabular}               & $A_P$        & $Z(L)_0$  \\ \hline
 \begin{tabular}{c} \\ \end{tabular}
                            & $A_P'\subset M_{P,\mathrm{ab}}$ & $L/[L,L]=Z(L)_0/Z(L)_0\cap[L,L]$ \\ \hline
                            & $X_*(A_P)$   & $\pi_1(Z(L)_0)$ \\ \hline
                   & $X_*(A_P')$  & $\pi_1(H)/\hat{\Lambda}_L =\pi_1(L/[L,L])$ \\ \hline
                   & $X_*(A_{P_0}')$ & $\pi_1(H)$\\ \hline
                   & $\fa_0=\fa_{P_0}$ & $\fh_\bR$\\ \hline
                   & \begin{tabular}{c}
                   $\fa_P =\bR \otimes X_*(A_P)$\\
                   $\quad =\bR\otimes X_*(A_P')$ \end{tabular}  & $(\fz_L)_\bR$ \\ \hline
                   & \begin{tabular}{c}
                   $\fa_G =\bR \otimes X_*(A_G)$\\
                   $\quad =\bR\otimes X_*(A_G')$ \end{tabular} & $(\fz_G)_\bR \cong \fh_\bR/V^*$ \\ \hline
& $\fa^G_0 =\fa^G_{P_0}\subset \fa_0$  &  $V^* =\Lambda\otimes \bR  \subset \fh_\bR$ \\ \hline
root system           & $\Phi_0=\Phi_{P_0} \subset \fa_0^\vee$   & $R\subset \fh_\bR^*$ \\ \hline
set of positive roots & $\Phi_0^+=\Phi_{P_0}^+\subset \Phi_0$    & $R_+\subset R$ \\ \hline
set of simple roots   & $\Delta_0 =\Delta_{P_0}\subset \Phi_0^+$ & $\Delta\subset R_+$ \\ \hline
coroot lattice of $G$ & $\bigoplus_{\alpha\in \Delta_0}\bZ \alpha^\vee$ &
$\Lambda = \bigoplus_{\alpha\in \Delta}\bZ \alpha^\vee \subset \pi_1(H)$ \\ \hline
\end{longtable}\end{center}

In this section, we will closely follow the notation in \cite{LR}. We will not repeat
most of the definitions in \cite{LR}. Following \cite{LR}, if $P\subset Q\subset R$
are three parabolic subgroups of $G$, there are canonical splittings
$\fa_P=\fa^Q_P\oplus\fa^R_Q\oplus\fa_R$ and
$\fa_P^* =\fa_P^{Q*}\oplus \fa_Q^{R*} \oplus \fa_R^*$.
Given $H\in\fa_P$, we denote by $[H]^Q$, $[H]_Q^R$, and $[H]_R$ the canonical projections
of $H$ onto $\fa^Q_P$, $\fa^R_Q$, and $\fa_R$, respectively. The components of
$\beta\in \fa_P^*$ in $\fa_P^{Q*}$, $\fa_Q^{R*}$, and $\fa_R^*$ are $\beta|_{\fa^Q_P}$,
$\beta|_{\fa^R_Q}$, and $\beta|_{\fa_R}$, respectively.
Given $\alpha\in \Delta_P=\Delta_P^G\subset \fa_P^{G*}$, let $\tal$ denote
the unique element in $\Delta_0\subset \fa_{P_0}^{G*}$ such that $\tal|_{\fa_P^G}=\alpha$.
Then $\tal^\vee\in \fa_{P_0}^G$ and
$\alpha^\vee = [\tal^\vee]_P \in \fa_P^G$.
The subset $I^P$ of the set of simple roots in \cite{FM} corresponds to $\Delta_P=\Delta_P^G$
in the following way:
\begin{eqnarray*}
I^P =  \{ \tilde{\alpha}\mid \alpha\in \Delta_P \} \subset \Delta_0 \subset \fa_{P_0}^{G*}, & &  \Delta_P = \{\beta|_{\fa_P}\mid \beta\in I^P\}\subset \fa_P^{G*}\\
\Delta_{P_0}^P = \{\beta|_{\fa_{P_0}^P}\mid \beta \in \Delta_0\setminus I^P\} \subset \fa_{P_0}^{P*}, & &  
\Delta_{P}^{Q}  = \{\beta|_{\fa_{P}^Q}\mid \beta \in I^P\setminus I^Q\}\subset \fa_{P}^{Q*}
\end{eqnarray*}

We continue the table of correspondence between notations in \cite{LR} and \cite{FM}:
\\
\begin{center}
\begin{tabular}{||c|c|c||}\hline
                     & \cite{LR} & \cite{FM} \\ \hline
                     & $X_*(A_P')$ & $\pi_1(H)/\hat{\Lambda}_L$ \\ \hline
                     & $\Lambda^G_P =X_*(A_P')/\bigoplus_{\alpha \in \Delta^G_P}\bZ \alpha^\vee$
                     & $\pi_1(H) / (\hat{\Lambda}_L\oplus  \bigoplus_{\alpha\in I^P} \bZ\alpha^\vee)$ \\ \hline
\begin{tabular}{c}
Topological type \\
 of $G$-bundle
\end{tabular}
                     & $\Lambda^G_{P_0}$  & $\pi_1(H)/\Lambda =\pi_1(G)$\\ \hline
\begin{tabular}{c}
Topological type \\
 of $M_P$-bundle
\end{tabular}
                     & $\Lambda^P_{P_0}$ & $\pi_1(H)/\Lambda_L =\pi_1(L)$ \\ \hline
\end{tabular}\end{center}

Recall that we have the commutative diagram \eqref{diagram_with_fund_weights}, 
in which the $\varpi_\alpha$ are the fundamental weights and whose columns and first row are exact. Given a parabolic subgroup $P$ of $G$, the topological
type of an $M_P$-bundle is given by $\nu_P\in \Lambda^P_{P_0}\cong \pi_1(M_P)$.
The slope of an $M_P$-bundle is given by $\nu_P'\in X_*(A_P')$.
So the commutative diagram \eqref{diagram_with_fund_weights} can be rewritten as follows:
\begin{equation}\label{diagram_with_fund_weights3}
\begin{CD}
@. 0     @.  0\\
@. @VVV @VVV\\
0 @>>> \hLambda_P/\Lambda_P @>{j_{ss}}>> \hLambda/\Lambda
@>{ \oplus_{\alpha\in \Delta^G_P}\varpi_\tal }>> \oplus_{\alpha\in \Delta^G_P}\bQ/\bZ \\
@.  @VV{i_P}V @VV{i_G}V @|\\
@.  \Lambda^P_{P_0} @>{[\cdot]_G}>> \Lambda^G_{P_0}
@>{\oplus_{\alpha\in \Delta^G_P}\varpi_\tal}>> \oplus_{\alpha\in \Delta^G_P}\bQ/\bZ \\
@. @VV{p_P}V @VV{p_G}V\\
@. X_*(A_P') @>{[\cdot]_G}>>  X_*(A_G') \\
@. @VVV @VVV\\
@. 0 @. 0 \\
\end{CD}
\end{equation}
Here $\Lambda_P =\oplus_{\alpha\in \Delta_{P_0}^P} \bZ\alpha^\vee\subset
 X_*(A_{P_0}')$, and $\hLambda_P$  is the saturation of $\Lambda_P$ in $X_*(A'_{P_0})$.
Let $\nu_P'$ and $\nu_G'$ denote the projections $p_P(\nu_P)$ and $p_G(\nu_G)$, respectively. Recall that $\{\varpi_\alpha \mid \alpha\in \Delta_0\}$ is a basis of
the real vector space $\fa_{P_0}^{G*}$ which is dual to the basis
$\{ \alpha^\vee \mid \alpha\in\Delta_0\}$ of $\fa_{P_0}^G$. Given $\alpha\in \Delta_0$,
we extend $\varpi_\alpha:\fa_{P_0}^G\lra \bR$ to $\varpi_\alpha:\fa_0=\fa_{P_0}^G \oplus \fa_G \lra \bR$ by
zero on $\fa_G$. Then $\varpi_\alpha$ takes integral values on
$\oplus_{\alpha\in \Delta_0} \bZ \alpha^\vee\subset \fa_{P_0}^G \subset \fa_0$,
and takes rational values on $X_*(A_{P_0}') \subset \fa_0$. So it induces a map
$$
\varpi_\alpha: \Lambda_{P_0}^Q=X_*(A_{P_0}')\left/\bigoplus_{\alpha\in \Delta_{P_0}^Q}\bZ\alpha^\vee\right.
\lra \bQ/\bZ
$$
where $Q$ is any parabolic subgroup of $G$. More explicitly,
given $\nu_Q \in \Lambda_{P_0}^Q$, let $X\in X_*(A'_{P_0})$ be a representative
of $\nu_Q$. Then $\varpi_\alpha(\nu_Q)= \varpi_\alpha(X)+\bZ$.

\subsubsection{Inversion formulas}
Let $A$ be a fixed topological abelian group.
In \cite{LR}, Laumon and Rapoport introduced the notion of $\widehat\Gamma$-converging
functions and $\Gamma$-converging functions from $\fP$ to $A$, where
$$
\fP=\{ (P,\nu_P')\mid P\in \cP, \nu'_P \in X_*(A_P') \}.
$$
We will introduce similar notion for
functions from $\fT$ to $A$, where
$$
\fT=\{ (P,\nu_P)\mid P\in \cP, \nu_P \in \Lambda^P_{P_0} \}.
$$

\begin{definition} Let
$\fT=\{ (P,\nu_P)\mid P\in \cP, \nu_P \in \Lambda^P_{P_0} \}$,
and let $A$ be a fixed topological abelian group.
A function $a: \fT\lra A$ is
{\em $\widehat\Gamma$--converging} if for each standard parabolic subgroup
$P\subset Q$ of $G$ and each $\nu_Q\in \Lambda^Q_{P_0}$, the finite sum
$$
\sum_{\scriptstyle \nu_P\in \Lambda^P_{P_0}\atop\scriptstyle [\nu_P]_Q=\nu_Q}
\widehat\Gamma_P^Q([\nu_P']^Q,T)a(P,\nu_P)
$$
admits a limit as $T\in \fa_P^{Q+}$ goes to infinity. If this is the case,
we shall denote this limit by
$$
\sum_{\scriptstyle \nu_P\in \Lambda^P_{P_0}\atop\scriptstyle [\nu_P]_Q=\nu_Q}
\widehat\tau_P^Q([\nu_P']^Q)a(P,\nu_P)\ .
$$
A function $b: \fT \lra A$ is {\em $\Gamma$--converging} if for each
standard parabolic subgroup $P\subset Q$ of $G$ and
each $\nu_Q\in \Gamma^Q_{P_0}$, the finite sum
$$
\sum_{\scriptstyle \nu_P\in \Lambda^P_{P_0}\atop\scriptstyle [\nu_P]_Q=\nu_Q}
\Gamma_P^Q([\nu_P']^Q,T)b(P,\nu_P)
$$
admits a limit as $T\in \fa_P^{Q+}$ goes to infinity.
If this is the case, we shall denote this limit by
$$
\sum_{\scriptstyle \nu_P\in \Lambda^P_{P_0}\atop\scriptstyle [\nu_P]_Q=\nu_Q}
\tau_P^Q([\nu_P']^Q)b(P,\nu_P)\ .
$$
\end{definition}

\begin{theorem}[{\cite[Theorem A.1]{Ho_Liu_YM_AMS}}]\label{thm:inversionI}
For each $\widehat\Gamma$--converging function 
$a:\fT\rightarrow A$, there exists a
unique $\Gamma$--converging function $b:\fT\rightarrow A$ such that,
for each $(Q,\nu_Q)\in \fT$, we have
$$
a(Q,\nu_Q)=\sum_{\scriptstyle P\in \cP\atop\scriptstyle P\subset Q}
\sum_{\scriptstyle \nu_P\in \Lambda^P_{P_0}\atop\scriptstyle [\nu_P]_Q=\nu_Q}
\tau_P^Q([\nu_P']^Q)b(P,\nu_P)\ .
$$

The function $b$ is given by the following formula~:
for each $(Q,\nu_Q)\in \fT$, we have
$$
b(Q,\nu_Q)=
\sum_{\scriptstyle P\in \cP\atop
\scriptstyle P\subset Q}(-1)^{{\rm dim}(\fa_P^Q)}
\sum_{\scriptstyle \nu_P\in \Lambda^P_{P_0}\atop\scriptstyle [\nu_P]_Q=\nu_Q}
\widehat\tau_P^Q([\nu_P']^Q)a(P,\nu_P)\ .
$$
\end{theorem}

Now we consider a special case of Theorem \ref{thm:inversionI}.
For any $P\in \cP$, fix $n_P\in \bZ_{\geq 0}$ and
$\ep^P_0 \in \fa_0^{P*}\subset \fa_0^*$ such that
for any standard parabolic subgroups $P\subset Q$ of $G$,
$$
n_P \geq n_Q,\quad
(\ep^Q_0 -\ep^P_0)\bigr|_{\fa_0^P}=0,
\quad \langle \ep^Q_P,\alpha^\vee\rangle\in \bZ_{>0}
\quad \forall \alpha\in \Delta_P^Q,
$$
where $\ep_P^Q=(\ep^Q_0-\ep^P_0)\bigr|_{\fa^Q_P}$ (here, we use $\ep_P^Q$ instead of the symbol $\delta_P^Q$  used in \cite{LR}).

Ho-Liu proved the following analogue of  \cite[Lemma 2.3]{LR}:
\begin{lemma}[{\cite[Lemma A.5]{Ho_Liu_YM_AMS}}] 
For each $(Q,\nu_Q)\in \fT$ and each standard parabolic subgroup
$P\subset Q$
of $G$, we have
$$\displaylines{
\qquad\sum_{\scriptstyle \nu_P\in \Lambda^P_{P_0}\atop
\scriptstyle [\nu_P]_Q=\nu_Q}
\widehat\tau_P^Q([\nu_P']^Q)t^{\langle\ep_P^Q,[\nu_P']^Q\rangle}
=\Bigl(\prod_{\alpha\in\Delta_P^Q}
{1\over 1-t^{\langle\ep_P^Q,\alpha^\vee\rangle}}\Bigl)
t^{\sum_{\alpha\in\Delta_P^Q}\langle\ep_P^Q,\alpha^\vee\rangle
\langle\varpi_\tal (\nu_Q)\rangle}\,,\qquad}
$$
where, for each $\mu \in \bR/\bZ$, $\langle \mu \rangle\in \bR$
is the unique representative of the class $\mu$ such that
$0<\langle \mu \rangle\leq 1$.
\end{lemma}

\begin{corollary}
For each $(Q,\nu_Q)\in \fT$ and each standard parabolic subgroup
$P\subset Q$
of $G$, we have
$$\displaylines{
\qquad\sum_{\scriptstyle \nu_P\in \Lambda^P_{P_0}\atop
\scriptstyle [\nu_P]_Q=\nu_Q}
\widehat\tau_P^Q([\nu_P']^Q)(uv)^{\frac{1}{2} \langle\ep_P^Q,[\nu_P']^Q\rangle}
=\Bigl(\prod_{\alpha\in\Delta_P^Q}
{1\over 1-(uv)^{\frac{1}{2}\langle\ep_P^Q,\alpha^\vee\rangle}}\Bigl)
(uv)^{\frac{1}{2}\sum_{\alpha\in\Delta_P^Q}\langle\ep_P^Q,\alpha^\vee\rangle
\langle\varpi_\tal (\nu_Q)\rangle}\,,\qquad}
$$
\end{corollary} 

Let $A$ be a $\bZ[\![u,v]\!]$-module equipped with the $\bZ[\![u,v]\!]$-adic topology and we assume
that $A$ is a complete with this topology.
Setting $m(P,\nu_P')= n_P + \langle \ep^G_P,\nu_P'\rangle$, we have now concluded with
the following refined version of the inversion formula \cite[Theorem A.6]{Ho_Liu_YM_AMS}
(which is a slightly modified version of \cite[Theorem 2.4]{LR}) for a $\bZ[\![t ]\!]$-module
which is complete with the $\bZ[\![t]\!]$-adic topology. 
\begin{theorem}\label{thm:inversionII}
Given $a_0:\cP\lra A$, there exists a unique function $b_0:\fT\rightarrow A$
which satisfies the relation
$$
a_0(Q)=\sum_{\scriptstyle P\in \cP \atop\scriptstyle P\subset Q}
\sum_{\scriptstyle \nu_P\in \Lambda^P_{P_0}\atop\scriptstyle [\nu_P]_Q=\nu_Q}
\tau_P^Q([\nu_P']^Q)b_0(P,\nu_P)(uv)^{\frac{1}{2}(m(P,\nu_P')-m(Q,\nu_Q'))}\,,
$$
for each $(Q,\nu_Q)\in\fT$. This function is given by
$$
b_0(Q,\nu_Q)=
\sum_{\scriptstyle P\in \cP \atop
\scriptstyle P\subset Q}
(-1)^{\dim(\fa_P^Q)}a_0(P)(uv)^{\frac{1}{2}(n_P-n_Q)}
\Bigl(\prod_{\alpha\in\Delta_P^Q}
{1\over 1-(uv)^{\frac{1}{2} \langle\ep_P^Q,\alpha^\vee\rangle}}\Bigl)\cdot
(uv)^{\frac{1}{2}\sum_{\alpha\in\Delta_P^Q}\langle\ep_P^Q,\alpha^\vee\rangle
\langle\varpi_\tal (\nu_Q)\rangle}
\in A\,,
$$ 
for each $(Q,\nu_Q)\in\fT$.
\end{theorem}

\subsubsection{Recursion}

Let $\cC(G,\nu_G)$ be the space of complex structures
on a $C^\infty$ principal $G$-bundle over a Riemann
surface of genus $g\geq 2$ with topological type $\nu_G \in \Lambda^G_{P_0} \cong\pi_1(G)$.
Let $\cC^{ss}(G,\nu_G)\subset \cC(G,\nu_G)$ be the semi-stable stratum.
Let $P_t(G,\nu_G)$ and $P_t^{ss}(G,\nu_G)$ be the $\cG$-equivariant
Poincar\'{e} series of $\cC(G,\nu_G)$ and $\cC^{ss}(G,\nu_G)$, respectively.
Let $\cC(G,P,\nu_P)\subset \cC(G,\nu_G)$ be the stratum which corresponds to
$(P,\nu_P)\in \fT$, where $[\nu_P]_G =\nu_G$. Then the real codimension $m(P,\nu_P')$
of the stratum
$\cC(G,P,\nu_P)$ is equal to
$$
 2\dim(N_P)(g-1) + 4 \langle\varrho_P^G,\nu_P'\rangle \in 2\bZ,
$$
where $N_P$ is the unipotent radical of $P$ and
$$
\varrho_P^G=\frac{1}{2}\sum_{\alpha\in \Phi_P^{G+} } \alpha\in \fa^{G*}_P
\subset \fa_P^*.
$$
Clearly $m(G,\nu_G') =0$. With the above notation, we have the following refinement of the  Atiyah-Bott recursion relation: 
\begin{theorem} \label{thm:recursion}
\begin{equation}\label{eqn:ab}
HP(G,\nu_G)(u,v)=\sum_{P\in \cP}
\sum_{\scriptstyle \nu_P\in \Lambda^P_{P_0}\atop
\scriptstyle [\nu_P]_G=\nu_G}
\tau_P^G([\nu_P']^G)(uv)^{ \frac{1}{2} m(P,\nu_P')}HP^{ss}(M_P,\nu_P)(u,v).
\end{equation}
where $HP(G,\nu_G)(u,v)= HP_{\fBun_{\, G_M,d}}(u,v)$ is given in Theorem \ref{thm:Hodge}.
\end{theorem}

Note that in both Theorem \ref{thm:recursion} and Theorem \ref{thm:Hodge}, we
may replace $G$ by the Levi component $M_P$ of a parabolic subgroup $P$.

\subsubsection{Inversion of the recursion} 

To invert the recursion relation \eqref{eqn:ab}, we
apply Theorem \ref{thm:inversionII}, with
$$
a_0(P)=HP(M_P,\nu_P)(u,v),
~b_0(P,\nu_P)=HP^{ss}(M_P,\nu_P)(u,v),
~n_P=2\dim(N_P)(g-1),
~\ep_P^G =4\varrho_P^G.
$$
We obtain
\begin{theorem}\label{thm:inversionIII}
For any $\nu_G\in \Lambda^G_{P_0}$, we have
$$
\begin{aligned}
HP^{ss}(G,\nu_G) =& \sum_{P\in \cP}
(-1)^{{\rm dim}(\fa_P^G)}
\Bigl({(1+u)^g(1+v)^g\over 1-uv}\Bigr)^{{\rm dim}(\fa_P)}
\Biggl(\prod_{i=1}^{{\rm dim}(\fa_0^P)}
{(1+u^{d_i(M_P)}v^{d_i(M_P)-1})^g(1+u^{d_i(M_P)-1} v^{d_i(M_P)})^g\over (1-(uv)^{d_i(M_P)-1})
(1-(uv)^{d_i(M_P)})}\Biggr)\\
&  \cdot (uv)^{{\rm dim}(N_P)(g-1)}
\Bigl(\prod_{\alpha\in\Delta^G_P}
{1\over 1-(uv)^{2\langle\varrho^G_P,\alpha^\vee\rangle}}\Bigl)
\cdot (uv)^{2\sum_{\alpha\in\Delta^G_P}\langle\varrho^G_P,\alpha^\vee\rangle
\langle\varpi_\tal (\nu_G)\rangle} \in \bQ(u,v).
\end{aligned}
$$
\end{theorem}
This is exactly Theorem \ref{thm:HP}.

\subsection{Classical groups}\label{section:classical_groups}

Let us write down explicitly the Hodge-Poincaré series of moduli stacks of semi\-stable $G$-bundles of degree $d$ on $X$ when $G$ is a classical almost simple group. We compute directly from the closed formula of Theorem \ref{thm:HP}, with the help of \cite[\S3.4]{Ho_Liu_YM_AMS} for the description of Harder-Narasimhan types. When we set $u=v=t$ in the formulas below, we recover the closed formulas for the Poincaré polynomials of moduli stacks of semistable $G$-bundles of degree $d$ on $X$. More precisely, when $G=\GL_r(\C)$, we recover the formula of Zagier in \cite{Zagier} and Laumon-Rapoport in \cite[Section 4]{LR}. And when $G=\SO_{2r+1}(\C)$, $\SO_{2r}(\C)$ or $\Sp_{r}(\C)$, we recover Theorems 5.5, 7.4 and 6.4 of \cite{Ho_Liu_YM_AMS}. 

\begin{remark}\label{Zagier_convention}
The symbol $\langle x\rangle$, for $x\in\Q$ is used in both \cite{Zagier} and \cite{LR}, but with a different definition. For instance, if $r$ and $d$ are integers such that $0\leq d < r$, then $\langle\frac{d}{r}\rangle:=1+\left\lfloor\frac{d}{r}\right\rfloor-\frac{d}{r}=1-\frac{d}{r}$ in \cite{Zagier}, which is equal to $\langle -\frac{d}{r}\rangle$ in \cite{LR}. Throughout the present paper, we use the Laumon-Rapoport convention, meaning that for all $x\in\Q$, the symbol $\langle x\rangle$ denotes the unique representative in $]0;1]$ of the class of $x$ in $\Q/\Z$.
\end{remark}

\smallskip

\begin{theorem}[Type A]\label{HP_for_GL_r}
When $G=\GL_r(\bC)$, we have, for all $d\in \bZ \simeq \pi_1(\GL_r(\bC))$, 
\begin{eqnarray*}
HP_{u,v}\Big(\fBun_{\GL_r(\C)_X}^{\ d,\,ss}\Big) & = & 
\sum_{\ell=1}^r (-1)^{\ell-1} \left(\frac{(1+u)^g(1+v)^g} {1-uv}\right)^\ell
\sum_{\tiny \begin{array}{c}r_1,\,\ldots\,, r_\ell\in\bZ_{>0}\\ r_1+\,\ldots\,+r_\ell = r \end{array}}
\prod_{i=1}^{\ell}  \prod_{k=2}^{r_i}
 \frac{(1+u^k v^{k-1})^g (1+u^{k-1} v^k)^g}{\big(1- (uv)^{k-1}\big)\big(1-(uv)^k\big)}
\\
&& 
\quad 
\times \frac{(uv)^{(g-1)\sum_{i<j} r_i r_j}}{\prod_{i=1}^{\ell-1}(1-(uv)^{r_i+r_{i+1}} ) }
\, (uv)^{\sum_{i=1}^{\ell-1} (r_i+r_{i+1}) \big\langle (r_1+\cdots + r_i)(-\frac{d}{r})\big\rangle}\ .
\end{eqnarray*}
And, when $G=\SL_r(\C)$, we have $\pi_1(\SL_r(\C))=0$ and \begin{eqnarray*}
HP_{u,v}\Big(\fBun_{\SL_r(\C)_X}^{\,ss}\Big) & = & 
\sum_{\ell=1}^r (-1)^{\ell} \left(\frac{(1+u)^g(1+v)^g} {1-uv}\right)^{\ell-1}
\sum_{\tiny \begin{array}{c}r_1,\,\ldots\,, r_\ell\in\bZ_{>0}\\ r_1+\,\ldots\,+r_\ell = r \end{array}}
\prod_{i=1}^{\ell}  \prod_{k=2}^{r_i}
 \frac{(1+u^k v^{k-1})^g (1+u^{k-1} v^k)^g}{\big(1- (uv)^{k-1}\big)\big(1-(uv)^k\big)}
\\
&& 
\quad \times \frac{(uv)^{(g-1)\sum_{i<j} r_i r_j} }{\prod_{i=1}^{\ell-1}(1-(uv)^{r_i+r_{i+1}} ) }
\, (uv)^{\sum_{i=1}^{\ell-1} (r_i+r_{i+1}) \big\langle (r_1+\cdots + r_i)(-\frac{d}{r})\big\rangle}\ .
\end{eqnarray*}
\end{theorem}

\medskip

\begin{theorem}[Type B]
When $G=\SO_{2r+1}(\bC)$, we have, for all $d \in \Z/2\Z \simeq \pi_1\SO_{2r+1}(\bC)$,
\begin{eqnarray*}
&& HP_{u,v}\Big(\fBun_{\SO_{2r+1}(\C)_X}^{\ d,\,ss}\Big)  \\
&=&\sum_{\ell=1}^r\sum_{\tiny \begin{array}{c}r_1,\,\ldots\,, r_\ell\in\bZ_{>0}\\ r_1+\,\ldots\,+r_\ell = r \end{array}}
\left( (-1)^\ell \prod_{i=1}^\ell \frac{\prod_{k=1}^{r_i} (1+u^k v^{k-1})^g (1+ u^{k-1}v^k)^g}
{ (1-(uv)^{r_i})\prod_{k=1}^{r_i-1}(1-(uv)^{k})^2 }\right. \\
&&\qquad \times \frac{(uv)^{(g-1)\sum_{i<j} (r_i r_j+r(r+1)/2 )} }
{\left[\prod_{i=1}^{\ell-1}(1-(uv)^{(r_i+r_{i+1} )} )\right](1-(uv)^{2r_\ell}) }
 (uv)^{\sum_{i=1}^{\ell-1}(r_i +r_{i+1})+ 2 r_\ell\langle d/2\rangle}\\
&&  + (-1)^{\ell-1} \prod_{i=1}^{\ell-1} \frac{\prod_{k=1}^{r_i}(1+u^k v^{k-1})^g (1+ u^{k-1}v^k)^g  }
{ (1-(uv)^{r_i})\prod_{k=1}^{r_i-1}(1-(uv)^k)^2 }
\cdot\frac{\prod_{k=1}^{r_\ell}(1+u^{2k}v^{2k-1})^g (1+u^{2k-1}v^{2k})^g }{ \prod_{k=1}^{2r_\ell}(1-(uv)^k)}\\
&& \qquad \left.\times
 \frac{(uv)^{(g-1)(\sum_{i<j} r_i r_j+r(r+1)/2 -r_\ell(r_\ell+1)/2)} }
{\left[\prod_{i=1}^{\ell-2}(1-(uv)^{ r_i+r_{i+1}} ) \right](1-\varepsilon(\ell)(uv)^{r_{\ell-1}+2r_\ell}) }(uv)^{\sum_{i=1}^{\ell-1}(r_i+ r_{i+1})+ \varepsilon(\ell)r_\ell }
\right)
\end{eqnarray*}
where
$$
\varepsilon(\ell)=\left\{\begin{array}{ll}0 & \mathrm{if}\ \ell=1\,,\\ 1 & \mathrm{if}\ \ell >1\,. \end{array}\right.
$$
\end{theorem}

\medskip

\begin{theorem}[Type C]\label{thm:typeC}
When $G=\Sp_r(\C)$, we have $\pi_1\Sp_r(\C)=0$ and 
\begin{eqnarray*}
&& HP_{u,v}\Big(\fBun_{\Sp_r(\C)_X}^{\ ss}\Big)  \\
&=&\sum_{\ell=1}^r\sum_{\tiny \begin{array}{c}r_1,\,\ldots\,, r_\ell\in\bZ_{>0}\\\sum r_i=r \end{array}}
\left( (-1)^\ell \prod_{i=1}^\ell \frac{\prod_{k=1}^{r_i}  (1+ u^k v^{k-1})^g (1+u^{k-1}v^k)^g}
{ (1- (uv)^{r_i})\prod_{k=1}^{r_i-1}(1-(uv)^k)^2 }\right. \\
&& \qquad \times \frac{(uv)^{(g-1)\sum_{i<j} (r_i r_j+r(r+1)/2)} }
{\left[\prod_{i=1}^{\ell-1}(1-(uv)^{r_i+r_{i+1} } )\right](1-(uv)^{r_\ell+1}) }
 (uv)^{ \sum_{i=1}^{\ell-1}(r_i +r_{i+1})+ (r_\ell+1)}\\
&& \quad + (-1)^{\ell-1} \prod_{i=1}^{\ell-1} \frac{\prod_{k=1}^{r_i}  (1+ u^k v^{k-1})^g (1+u^{k-1}v^k)^g  }
{ (1-(uv)^{r_i})\prod_{k=1}^{r_i-1}(1-(uv)^k)^2 }
\cdot\frac{\prod_{k=1}^{r_\ell}(1+ u^{2k} v^{2k-1})^g (1+u^{2k-1}v^{2k})^g  
}{\prod_{r=1}^{2r_\ell}(1-(uv)^{k})}\\
&& \qquad\left.\times
 \frac{(uv)^{(g-1)(\sum_{i<j} r_i r_j+r(r+1)/2-r_\ell(r_\ell+1)/2)}}
{\left[\prod_{i=1}^{\ell-2}(1-(uv)^{(r_i+r_{i+1})} ) \right](1-\varepsilon(\ell) (uv)^{r_{\ell-1}+ 2r_\ell+1}) }
(uv)^{\sum_{i=1}^{\ell-2}(r_i+ r_{i+1})+ \varepsilon(\ell)(r_{\ell-1}+ 2r_\ell+1) }
\right)
\end{eqnarray*}
where
$$
\varepsilon(\ell)=\left\{\begin{array}{ll}0 & \mathrm{if}\ \ell=1\, ,\\ 1 & \mathrm{if}\ \ell>1\,. \end{array}\right.
$$
\end{theorem} 

\medskip

\begin{theorem}[Type D]\label{thm:typeD}
When $G=\SO_{2r}(\C)$ (with $r>1$), we have, for all $d \in \Z/2\Z\simeq\pi_1\SO_{2r}(\C)$,
\begin{eqnarray*}
&& HP_{u,v}\Big(\fBun_{\SO_{2r}(\C)_X}^{\ d,\,ss}\Big) \\
&=&  \sum_{\ell=2}^r\sum_{\tiny \begin{array}{c}r_1,\,\ldots\,, r_\ell\in\bZ_{>0}\\ \sum r_i=r, r_\ell=1 \end{array}}
(-1)^\ell \prod_{i=1}^\ell \frac{\prod_{k=1}^{r_i} (1+u^k v^{k-1})^g (1+u^{k-1}v^k)^g}
{ (1-(uv)^{r_i})\prod_{k=1}^{r_i-1}(1- (uv)^{k})^2 }\\
&& \qquad \times
 \frac{(uv)^{(g-1)(\sum_{i<j} r_i r_j+r(r-1))} }
{\left[\prod_{i=1}^{\ell-1}(1-(uv)^{r_i+r_{i+1}} ) \right](1-(uv)^{(r_{\ell-1}+1)}) }
(uv)^{\sum_{i=1}^{\ell-2}(r_i+ r_{i+1})+ 2(r_{\ell-1}+1)\langle d/2\rangle }\\
&&+ \sum_{\ell=1}^{r-1}
\sum_{\tiny \begin{array}{c}r_1,\,\ldots\,, r_\ell\in\bZ_{>0}\\ \sum r_i=r, r_\ell>1 \end{array}}
\left(2(-1)^\ell \prod_{i=1}^\ell \frac{\prod_{k=1}^{r_i}  (1+u^k v^{k-1})^g (1+u^{k-1}v^k)^g }
{ (1-(uv)^{r_i})\prod_{k=1}^{r_i-1}(1-(uv)^{k})^2 } \right.\\
&& \qquad\times \frac{(uv)^{(g-1)(\sum_{i<j} r_i r_j+r(r-1))} }
{\left[\prod_{i=1}^{\ell-1}(1-(uv)^{r_i+r_{i+1}} )\right](1-(uv)^{2(r_\ell-1)}) }
(uv)^{\sum_{i=1}^{\ell-1}(r_i +r_{i+1})+2 (r_\ell-1)\langle d/2\rangle}\\
&& +(-1)^{\ell-1} \prod_{i=1}^{\ell-1} \frac{\prod_{k=1}^{r_i} (1+u^k v^{k-1})^g (1+u^{k-1}v^k)^g }
{ (1-(uv)^{r_i})\prod_{k=1}^{r_i-1}(1-(uv)^{k})^2 }
\cdot\frac{(1+(uv)^{2r_\ell-1})^{2g}\prod_{k=1}^{r_\ell-1}(1+u^{2k} v^{2k-1})^g (1+u^{2k-1}v^{2k})^g}
{(1-(uv)^{r_\ell-1})(1-(uv)^{r_\ell})\prod_{k=1}^{2r_\ell-2}(1-(uv)^{k})}\\
&&\qquad \left. \times
 \frac{(uv)^{(g-1)(\sum_{i<j} r_i r_j+r(r-1)/2 -r_\ell(r_\ell-1)/2)} }
{\left[\prod_{i=1}^{r-2}(1-(uv)^{r_i+r_{i+1}} ) \right](1-\varepsilon(\ell)(uv)^{r_{\ell-1}+2r_\ell-1}) }
(uv)^{\sum_{i=1}^{\ell-2}(r_i+ r_{i+1})+ \varepsilon(\ell)(r_{\ell-1}+2r_\ell-1) }\right)
\end{eqnarray*}
where
$$
\varepsilon(\ell)=\left\{\begin{array}{ll}0 & \mathrm{if} r=1\,,\\ 1 & \mathrm{if}\ r>1\,. \end{array}\right.
$$
\end{theorem} 

\section{Applications to moduli varieties of principal bundles}\label{applications}

In this section, we assume that the curve $X$ is of genus $g\geq 2$ and give a few applications of our results. Recall that, when $G$ is a connected complex reductive group, there exist irreducible quasi-projective algebraic varieties $\ModGXdss$ and $\ModGXdst$ parameterizing, respectively, families of $S$-equivalence classes of semistable $G$-bundles of degree $d$ on $X$ and families of isomorphism classes of stable $G$-bundles of degree $d$ on $X$. Moreover, the variety  $\ModGXdss$ is in fact projective, and contains $\ModGXdst$ as a dense, smooth open subset. Both were constructed by Ramanathan in \cite{Ramanathan}. If we consider the stack of stable $G$-bundles of degree $d$ on $X$, the canonical morphism 
\begin{equation}\label{CMS}
\BunGXdst \lra \ModGXdst
\end{equation} sending the groupoid $\BunGXdst(S)$ to its set of isomorphism classes of objects $\ModGXdst(S)$ makes the scheme $\ModGXdst$ a coarse moduli space for the stack $\BunGXdst$. Since the automorphism group of a stable $G$-bundle reduces to $Z_G$ by \cite{Ramanathan}, the morphism \eqref{CMS} is in fact a $Z_G$-gerbe, where $Z_G$ is the center of $G$.

\smallskip

In the vector bundle case, the so-called \textit{good case} is when the rank and degree are coprime: one can then use the morphism in \eqref{CMS} to deduce properties of the moduli space from properties of the moduli stack. We note that $r$ and $d$ are coprime if and only if every semistable vector bundle of rank $r$ and degree $d$ is stable. Thus, for principal bundles, we can look at the set of $d\in\piG$ such that   semistable principal $G$-bundles of degree $d$ on $X$ are in fact stable. As an example, consider the case when $G \simeq \GL_{r_1}(\C) \times \ldots\ \times \GL_{r_k}(\C)$ and $d = (d_1,\ldots, d_k)$ with $\frac{d_1}{r_1} = \ldots = \frac{d_k}{r_k}$ and $r_i \wedge d_i = 1$ for all $i$. For such a $d$, every semistable $G$-bundle of degree $d$ is indeed stable (see Remark \ref{sst_for_products}). To construct more examples, one can use the following criterion.

\begin{proposition}
Let $G$ be a connected complex reductive group and let $d \in \piG$. Then the following two conditions are equivalent:
\begin{enumerate}
	\item There exists a semistable principal $G$-bundle of degree $d$ which is not stable.
	\item There exists a Levi subgroup $L \subsetneq G$ and $d'\in \pi_1 L$ such that  $j_L(d') = d$ and
	$\mu_L(d') = \mu_G(d)\in \fh_\bR$, where $j_L: \pi_1 L\to \pi_1 G$ is induced by the inclusion 
	$L\to G$ and $\mu_L : \piL \to \mathfrak{h}_\R$ Equation \eqref{eqn:muL}.
	\end{enumerate}
\end{proposition}

\begin{proof}
Recall that a semistable principal $G$-bundle $E$ is strictly semistable if and only if it admits a reduction of structure group to a proper maximal Levi subgroup $L$ such that the reduced bundle $E_L$ is stable as an $L$-bundle. For such a subgroup $L$, Condition (2) is satisfied. Conversely, given a Levi subgroup $L \subset G$ satisfying Condition (2), the extension of structure group of a stable $L$-bundle provides a semistable $G$-bundle which is not stable.
\end{proof}  

\begin{remark}\label{good_case}
	We note that if $\piG = 0$, then for all $L \subset G$ and all $d' \in \pi_1 L$, we have  $j_L(d') = 0$ and $\mu_L(d') = \mu_G(0) = 0$. So when $G$ is simply connected, there exist semistable bundles which are not stable.
\end{remark}

In the good case (meaning, when every semistable principal $G$-bundle of degree $d \in \pi_1G$ is in fact stable), the $Z_G$-gerbe $\BunGXdss \lra \ModGXdss$ is cohomologically trivial (\cite{AB}), meaning that
\begin{equation}\label{cohom_trivial_gerbe}
H^*\big(\BunGXdss;\C\big) \simeq H^*\big(B Z_G;\C\big) \otimes H^*\big(\ModGXdss;\C\big)\,,
\end{equation} Such a gerbe is also called a neutral gerbe (\cite[Lemma~3.10]{Heinloth_lectures}). This makes it possible to compute the Poincaré polynomial of the moduli variety $\ModGXdss$ in this case. Indeed, by \eqref{classif_stack_max_torus}, one has $P_t(B Z_G) = \frac{1}{(1-t^2)^m}$ where $m=\dim Z_G$, so 
\begin{equation}\label{Poincare_polynomial_moduli_space}
P_t(\ModGXdss) = (1-t^2)^m P_t(\BunGXdss)\,,
\end{equation} where the Poincaré series $P_t(\BunGXdss)$ is computed either recursively or via the closed formula of Theorem \ref{thm:HP}. Recall that the isomorphism of $\C$-vector spaces \eqref{cohom_trivial_gerbe} is obtained by applying the Leray-Hirsch theorem to the fibration $$BZ_G \lra \BunGXdss \lra \ModGXdss$$ since, in the good case, the morphism $H^k(\BunGXdss;\C)\lra H^k(BZ_G;\C)$ induced by the inclusion of the fibre is surjective. By \cite{Deligne_Hodge_III}, Leray-Hirsch isomorphisms are compatible with Hodge structures, so, in the good case, we can compute the Hodge-Poincaré polynomial of the moduli variety $\ModGXdss$ using \eqref{cohom_trivial_gerbe}. More precisely, we get the following refinement of Formula \eqref{Poincare_polynomial_moduli_space}.

\begin{proposition}\label{HP_polyn_of_mod_space}
When all semistable principal $G$-bundles of degree $d$ on $X$ are stable, the Hodge-Poincaré polynomial of the moduli variety $\ModGXdss$ is given by the formula
$$HP_{u,v}(\ModGXdss) = (1-uv)^m HP_{u,v}(\BunGXdss)\,,$$ where $m=\dim Z_G$ and $HP_{u,v}(\BunGXdss)$ is computed either recursively via \eqref{recursive_formula_for_HP_series} or directly via the closed formula of Theorem \eqref{thm:HP}.
\end{proposition}

We can then deduce from Theorem \ref{HP_for_GL_r} and Proposition \ref{HP_polyn_of_mod_space} a closed formula for Hodge-Poincaré polynomials of moduli spaces of vector bundles of coprime rank and degree, and the fibres of the determinant map on those spaces. Of course, as noted in \cite{Zagier} and \cite{LR}, it is not clear from the formulas themselves that these expressions are indeed polynomials.

\begin{theorem}\label{HP_pol_mod_spaces_of_vector_bundles}
Let $r\wedge d=1$ and let $\ModXrd$ be the moduli space of semistable vector bundles of rank $r$ and degree $d$ on $X$. Let $\mathscr{L}\in\mathrm{Pic}_X^d$ be a line bundle of degree $d$ and set $\ModXrL:=\det^{-1}(\mathscr{L})$, where $\det:\ModXrd\lra\mathrm{Pic}_X^d$ is the morphism sending a vector bundle $\mathscr{E}$ to the line bundle $\det(\mathscr{E})$. Then the Hodge-Poincaré polynomials of these moduli spaces are given by the following formulas:
\begin{enumerate}
\item $HP_{u,v}(\ModXrd) = $ 
\begin{eqnarray*}
& (1+u)^g(1+v)^g & 
 \sum_{\ell=1}^r (-1)^{\ell-1} \left(\frac{(1+u)^g(1+v)^g} {1-uv}\right)^{\ell-1}
\sum_{\tiny \begin{array}{c}r_1,\,\ldots\,, r_\ell\in\bZ_{>0}\\ r_1+\,\ldots\,+r_\ell = r \end{array}}
\prod_{i=1}^{\ell}  \prod_{k=2}^{r_i}
 \frac{(1+u^k v^{k-1})^g (1+u^{k-1} v^k)^g}{\big(1- (uv)^{k-1}\big)\big(1-(uv)^k\big)}
\\
&& 
\quad 
\times \frac{(uv)^{(g-1)\sum_{i<j} r_i r_j}}{\prod_{i=1}^{\ell-1}(1-(uv)^{r_i+r_{i+1}} ) }
\, (uv)^{\sum_{i=1}^{\ell-1} (r_i+r_{i+1}) \big\langle (r_1+\cdots + r_i)(-\frac{d}{r})\big\rangle}\ .
\end{eqnarray*}
\item $HP_{u,v}(\ModXrL)= $
\begin{eqnarray*}
&& 
\sum_{\ell=1}^r (-1)^{\ell-1} \left(\frac{(1+u)^g(1+v)^g} {1-uv}\right)^{\ell-1}
\sum_{\tiny \begin{array}{c}r_1,\,\ldots\,, r_\ell\in\bZ_{>0}\\ r_1+\,\ldots\,+r_\ell = r \end{array}}
\prod_{i=1}^{\ell}  \prod_{k=2}^{r_i}
 \frac{(1+u^k v^{k-1})^g (1+u^{k-1} v^k)^g}{\big(1- (uv)^{k-1}\big)\big(1-(uv)^k\big)}
\\
&& 
\quad \times \frac{(uv)^{(g-1)\sum_{i<j} r_i r_j} }{\prod_{i=1}^{\ell-1}(1-(uv)^{r_i+r_{i+1}} ) }
\, (uv)^{\sum_{i=1}^{\ell-1} (r_i+r_{i+1}) \big\langle (r_1+\cdots + r_i)(-\frac{d}{r})\big\rangle}\ .
\end{eqnarray*}
\end{enumerate}
\end{theorem}

\begin{proof}
Since we are assuming that $r\wedge d=1$, the formula for $HP_{u,v}(\ModXrd)$ follows directly from Theorem \ref{HP_for_GL_r} and Proposition \ref{HP_polyn_of_mod_space} . Note that we have factored out the product $(1+u)^g(1+v)^g$, which coincides with $HP_{u,v}(\mathrm{Pic}_X^d)$. To obtain the formula for $HP_{u,v}(\ModXrL)$, recall from \cite[Proposition~9.7]{AB} that the fibration $\det: \ModXrd \lra \Pic_X^d$ is cohomologically trivial. So, using \cite{Deligne_Hodge_III}, the Hodge-Poincaré series of the moduli space $\ModXrL:=\det^{-1}(\mathscr{L})$ of semistable vector bundles of rank $r$ and determinant $\mathscr{L}$ satisfies $$HP_{u,v}(\ModXrL)\, HP_{u,v}(\Pic_X^d) = HP_{u,v}(\ModXrd)\,,$$ which immediately gives the second formula.
\end{proof}

\begin{example}\label{example_rank_2_vb_bis}
Let us retake Example \ref{example_rank_2_vb} and consider the moduli variety $\mathscr{M}_X^{2,1}$ of semistable rank $2$ vector bundles of degree $1$ on $X$. From Formula \ref{towards_HP_polyn_of_mod_space} and Proposition \ref{HP_polyn_of_mod_space}, we deduce:
\begin{eqnarray*}
HP_{u,v}\big(\mathscr{M}_X^{2,1}\big) & = & (1-uv)\, HP_{u,v}\left(\fBun_{\GL(2;\C)}^{\, 1,ss}\right) \\ 
& = & (1-uv)\, \left[\frac{(1+u)^g(1+v)^g}{1-uv}\ \frac{(1+u^2v)^g(1+uv^2)^g}{(1-uv)(1-u^2v^2)} - \frac{(uv)^g}{1-u^2v^2}\,\left(\frac{(1+u)^g\,(1+v)^g}{1-uv}\right)^2\right]\,.\\
& = & \frac{(1+u)^g(1+v)^g}{(1-uv)(1-u^2v^2)}\,\left[ \big((1+u^2v)(1+uv^2)\big)^g - \big(uv(1+u)(1+v)\big)^g\right]\\
& = &(1+u)^g\,(1+v)^g\,\sum_{k=0}^{g-1}\big((1+u^2v)(1+uv^2)\big)^{g-1-k}\,\big(uv(1+u)(1+v)\big)^k\, 
\end{eqnarray*} which is indeed a polynomial.
Let us fix a line bundle $\mathscr{L}\lra X$ of degree $1$ over $X$. Using the formula just proved, the fact that $HP_{u,v}(\mathrm{Pic}^1_X) =(1+u)^g\,(1+v)^g$, and the cohomologically trivial fibration $\mathscr{M}_X^{2,1}\lra \mathrm{Pic}^1_X$, we obtain the Hodge-Poincaré polynomial of the moduli space of semistable vector bundles of rank $2$ and determinant $\mathscr{L}$, which in small rank was already computed in \cite[Corollary~5]{Earl_Kirwan}:
\begin{equation}\label{rk2_vector_bundles_with_fixed_det}
HP_{u,v}\big(\mathscr{M}_X^{2,\mathscr{L}}\big) = \sum_{k=0}^{g-1}\big((1+u^2v)(1+uv^2)\big)^{g-1-k}\,\big(uv(1+u)(1+v)\big)^k\,.
\end{equation}
\end{example}

From Theorem \ref{HP_pol_mod_spaces_of_vector_bundles}, we can also compute a specialization of the Hodge-Poincaré polynomial called the $\chi_t$-characteristic and defined by $\chi_t(-) = HP_{-1,t}(-)$. This gives a new proof of the following result due to Earl and Kirwan in \cite[Corollary~6]{Earl_Kirwan}.

\begin{corollary}
Assume $r\wedge d=1$ and let $\mathscr{L}$ be a line bundle of degree $d$ on $X$. Then $HP_{-1,t}(\ModXrd) \equiv 0$ and
$$\chi_t(\ModXrL):=HP_{-1,t}(\ModXrL) = \prod_{k=2}^r \big(1-(-t)^{k-1}\big)^{g-1} \big(1-(-t)^k\big)^{g-1}\,.$$
\end{corollary}

\begin{proof}
The result follows directly from Theorem \ref{HP_pol_mod_spaces_of_vector_bundles}. For the first one, the product $(1+u)^g(1+v)^g$ is $0$ when $u=-1$. And for the second one, the only non-vanishing term in the sum is for $\ell=1$. So $r_1=r$ and the rest follows.
\end{proof}

As a consequence, one recovers the following results, the first one being due to Narasimhan-Ramanan in \cite{NaRa75} and the second to Earl-Kirwan in \cite{Earl_Kirwan}.

\begin{corollary}
Assume $r\wedge d=1$ and let $\mathscr{L}$ be a line bundle of degree $d$ on $X$. Then the Euler characteristic and signature of $\ModXrL$ are both zero.
\end{corollary}

\begin{proof}
The Euler characteristic of $\ModXrL$ is $$\chi(\ModXrL)=P_{-1}(\ModXrL)=HP_{-1,-1}(\ModXrL)=\chi_{-1}(\ModXrL) = 0$$ and, by the Hodge index theorem, its signature is $$\sigma(M) = HP_{-1,1}(\ModXrL) = \chi_1(\ModXrL)=0\,,$$ which concludes the proof.
\end{proof}


\def\cprime{$'$}

\end{document}